\documentclass[12pt, A4paper, oneside]{article}
\usepackage{amsthm,amsmath,amssymb,amscd,verbatim,epsfig,verbatim}
\usepackage{enumitem,multirow}
\theoremstyle{plain}
\newtheorem{thm}{Theorem}[section]
\newtheorem{lem}{Lemma}[section]
\newtheorem{col}{Corollary}[section]

\newtheorem{rem}{Remark}[section]

\numberwithin{equation}{section}

\begin{document}

\fontsize{12}{25pt}\selectfont

\title{One-sided solutions for optimal stopping problems with logconcave reward functions}

\author{Yi-Shen Lin\thanks{Institute of Statistical Science, Academia Sinica, Taipei 115, Taiwan, R.O.C. Email address: yslin@stat.sinica.edu.tw}\;\;and Yi-Ching Yao\thanks{Institute of Statistical Science, Academia Sinica, Taipei 115, Taiwan, R.O.C. Email address: yao@stat.sinica.edu.tw}
    \date{}
    } \maketitle
\centerline {October 11, 2017}
\begin{abstract}
\fontsize{12}{18pt}\selectfont
In the literature on optimal stopping, the problem of maximizing the expected discounted reward over all stopping times has been explicitly solved for some special reward functions (including $(x^+)^{\nu}$, $(e^x-K)^+$, $(K-e^{-x})^+$, $x\in\mathbb{R}$, $\nu\in(0,\infty)$ and $K>0$) under general random walks in discrete time and L\'evy processes in continuous time (subject to  mild integrability conditions). All of such reward functions are continuous, increasing and logconcave while the corresponding optimal stopping times are of threshold type (i.e. the solutions are one-sided). In this paper, we show that all optimal stopping problems with increasing, logconcave and right-continuous reward functions admit one-sided solutions for general random walks and L\'evy processes. We also investigate in detail the principle of smooth fit for L\'evy processes when the reward function is increasing and logconcave.
\end{abstract}
\emph{Keywords:} optimal stopping; threshold type; one-sided solution; logconcave function; random walk; L\'evy process; smooth fit.
\begin{tabbing}
2010 Mathematics Subject Classification:\;\=Primary 60G40\\
\>Secondary 62L15; 60J10; 60J65
\end{tabbing}

\fontsize{12}{20pt}\selectfont
\section{Introduction}
\hspace*{18pt}
Let $X=\{X_t\}_{t\ge0}$ be a process with independent stationary increments where the time parameter $t$ is either discrete (i.e. $t\in\mathbb{Z}^+=\{0,1,\dots\}$) or continuous (i.e. $t\in\mathbb{R}^+=[0,\infty)$). Let $X_0=x$ be the initial state. We assume that $X$ is defined on a filtered probability space $(\Omega,\mathcal{F},\{\mathcal{F}_t\},P)$, where for each $t$, $\mathcal{F}_t$ is the (enlarged) $\sigma$-field generated by $\{X_s:s\le t\}$. For a given nonnegative measurable reward function $g$ and a discount factor $q\ge0$, we study the optimal stopping problem of finding $\tau^*\in\mathcal{M}$ which satisfies
\begin{equation}\label{e1.1}
E_x\left(e^{-q\tau^*}g(X_{\tau^*})\mathbf{1}_{\{\tau^*<\infty\}}\right)=V(x):=\sup_{\tau\in\mathcal{M}}E_x\left(e^{-q\tau}g(X_{\tau})\mathbf{1}_{\{\tau<\infty\}}\right),
\end{equation}
where the subscript $x$ in $E_x$ refers to the initial state $X_0=x$, $\mathcal{M}$ is the collection of all stopping times taking values in $[0,\infty]$  and $\mathbf{1}_{A}$ denotes the indicator function of $A$. For $a\in \mathbb{R} \cup \{-\infty\}$, the stopping time $\tau_a:=\inf\{t\ge 0: X_t\ge a\}$ is said to be of threshold type with threshold $a$. An optimal stopping time of threshold type exists if $\tau^*=\tau_a$ for some $a\in \mathbb{R} \cup \{-\infty\}$, in which case the optimal stopping problem (\ref{e1.1}) is said to admit a one-sided solution.

In the literature, the problem (\ref{e1.1}) has been solved explicitly for special reward functions including $(x^+)^\nu=\left(\max\{x,0\}\right)^\nu$ ($\nu>0$), $(e^x-K)^+$ and $(K-e^{-x})^+$ in both discrete and continuous time. See Dubins and Teicher \cite{ref2}, Darling, Liggett and Taylor \cite{ref1}, Mordecki \cite{ref6}, Novikov and Shiryaev \cite{ref8} \cite{ref9}, and Kyprianou and Surya \cite{ref4}. Note that all of the above reward functions are continuous, increasing and logconcave, while the corresponding optimal stopping times are of threshold type. (Here and below, the word ``increasing" means ``nondecreasing.") Note also that for each of these reward functions, the one-sided solution for (\ref{e1.1}) is found explicitly under {\em general} random walks in discrete time and L\'{e}vy processes in continuous time (subject to mild integrability conditions). On the other hand, by imposing more structures on $\{X_t\}$, the problem (\ref{e1.1}) can be solved explicitly for more general reward functions. Indeed, assuming that $\{X_t\}$ is a matrix-exponential jump diffusion, Sheu and Tsai \cite{ref14} have obtained an explicit one-sided solution of (\ref{e1.1}) for a fairly general class of increasing and logconcave reward functions $g\ge0$ which satisfy some additional technical conditions. In view of the above results, two natural questions arise concerning the relationship between the logconcavity and monotonicity of $g$ and the existence of a one-sided solution: (Q1) Does the logconcavity and monotonicity of $g$ imply the existence of a one-sided solution? More precisely, if $g\ge0$ is increasing and logconcave, does (\ref{e1.1}) admit a one-sided solution under {\em general} random walks in discrete time and L\'{e}vy processes in continuous time? (Q2) To what extent is the logconcavity and monotonicity of $g$ implied by the existence of a one-sided solution? To make (Q2) more precise, observe that it is easy to find a reward function $g\ge0$ (which is neither increasing nor logconcave) and a process $\{Y_t\}_{t\ge0}$ with $Y_0=0$ such that for some threshold $a$, $\tau_a$ is optimal under $\{X_t=x+Y_t\}_{t\ge0}$ for any initial state $x\in\mathbb{R}$. Indeed, if for some $g\ge0$, $\tau_a=\inf\{t\ge0:X_t\ge a\}$ is optimal under $\{X_t=x+Y_t\}_{t\ge0}$ for all $x\in\mathbb{R}$, then with respect to the same process $\{X_t\}$,
\begin{align*}
E_x\left(e^{-q\tau_a}\tilde{g}(X_{\tau_a})\mathbf{1}_{\{\tau_a<\infty\}}\right)&=E_x\left(e^{-q\tau_a}g(X_{\tau_a})\mathbf{1}_{\{\tau_a<\infty\}}\right)\\
&=\sup_{\tau\in\mathcal{M}}E_x\left(e^{-q\tau}g(X_{\tau})\mathbf{1}_{\{\tau<\infty\}}\right)\\
&\ge \sup_{\tau\in\mathcal{M}}E_x\left(e^{-q\tau}\tilde{g}(X_{\tau})\mathbf{1}_{\{\tau<\infty\}}\right),
\end{align*}
where $\tilde{g}\ge0$ is any reward function satisfying
\begin{equation}\label{e1.2}
0\le\tilde{g}(x)\le g(x)\;\;\mbox{for all}\;\;x\in\mathbb{R}\;\;\mbox{and}\;\;\tilde{g}(x)=g(x)\;\;\mbox{for}\;\;x\ge a.
\end{equation}
This shows that, with respect to the process $\{X_t\}$ (for any $X_0=x\in\mathbb{R}$), $\tau_a$ is optimal for any reward function $\tilde{g}$ satisfying (\ref{e1.2}), which need not be increasing or logconcave. Thus, it seems natural to formulate (Q2) as ``If $g\ge0$ is such that (\ref{e1.1}) admits a one-sided solution with respect to a sufficiently rich class of processes $\{X_t\}$, is $g$ necessarily increasing and logconcave?" The work of Hsiau, Lin and Yao \cite{ref3} makes an attempt to address (Q1) and (Q2). Specifically, it is shown (\textit{cf.} \cite[Theorem 3.1]{ref3}) that if $g\ge0$ is increasing, logconcave and right-continuous, (\ref{e1.1}) admits a one-sided solution provided that $\{X_t\}$ is a spectrally negative L\'{e}vy process (for which a tractable fluctuation theory is available due to no overshoots). It is also shown by example that the right-continuity condition cannot be removed in general. Furthermore, it is established (\textit{cf.} \cite[Theorem 6.1]{ref3}) that for fixed $q>0$, a nonnegative measurable reward function $g$ is necessarily increasing and logconcave if for {\em each} $\alpha\in\mathbb{R}$, there is a threshold $u(\alpha)\in\mathbb{R}\cup\{-\infty\}$ such that $\tau_{u(\alpha)}$ is optimal with respect to the Brownian motion process $X_t=x+\alpha t+B_t$ where $\{B_t\}$ is standard Brownian motion. (A similar result is also established for the case $q=0$ (\textit{cf.} \cite[Theorem 5.2]{ref3}) where the drift parameter $\alpha$ is restricted to $\alpha\in(-\infty,0)$.)

The present paper addresses (Q1) in full generality (with the right-continuity condition imposed on $g$). Specifically, we treat the discrete-time case in Section 2, and show that if $g\ge0$ is nonconstant, increasing, logconcave and right-continuous, then with respect to any random walk, there is a unique threshold $-\infty\le u\le\infty$ such that $V(x)>g(x)$ for $x<u$ and $V(x)=g(x)$ for $x\ge u$. Moreover,
if $-\infty\le u<\infty$, $\tau_u$ is optimal attaining the (finite) value $V(x)$ of (\ref{e1.1}) for all $x\in\mathbb{R}$. If $u=\infty$, {\em either} $V(x)=\infty$ for all $x\in\mathbb{R}$ in which case there are randomized stopping times with an infinite expected (discounted) reward {\em or} $V(x)<\infty$ for all $x\in\mathbb{R}$ in which case no optimal stopping time exists. Via the standard time discretization device, the results in Section 2 are applied in Section 3 to general L\'{e}vy processes in continuous time.
With the help of the results in Section 3, we investigate the principle of smooth fit for L\'{e}vy processes $\{X_t\}$ when $g\ge0$ is a general increasing and logconcave function. Alili and Kyprianou \cite{ref18} have shown for $g(x)=(K-e^{-x})^+$ (described here in our setting which corresponds to perpetual American put) that the smooth fit principle holds if and only if $0$ is regular for $(0,\infty)$ for $\{X_t\}$. They have also conjectured that this result holds more generally. (See also Boyarchenko and Levendorski\v{i} \cite{ref19} for a related discussion and Peskir \cite{ref20} for an example in which $\{X_t\}$ is a regular diffusion process and $g$ is differentiable but the value function fails to satisfy the smooth fit condition at the optimal stopping boundary.) We show in Section 4 that their conjecture is true for general increasing and logconcave $g\ge0$ provided $\log g$ is not linear in any interval. If $\log g$ is linear in some interval, the smooth fit principle may hold even when $0$ is irregular for $(0,\infty)$ for $\{X_t\}$. Section 5 contains concluding remarks along with a discussion of conditions for the optimal threshold $u<\infty$. The proofs of some technical lemmas (stated in Section 2) are relegated to Section 6, which involve delicate arguments to deal with the issue of overshoots. It should be remarked that the optimal threshold and corresponding value function derived in Sections 2 and 3 are not explicit, which depend on (general) $g$ and $\{X_t\}$ in a complicated way. This fact makes it a challenging task to verify that the value function is excessive.

We close this section by briefly reviewing some recent papers in which effective methods are proposed to construct an {\em explicit} solution of (\ref{e1.1}) for $g$ (not necessarily increasing or logconcave) under $\{X_t\}$ which is a L\'{e}vy process or a more general Markov process in continuous time. Surya \cite{ref11} has introduced an averaging problem (associated with (\ref{e1.1})) whose solution, if it exists, yields a fluctuation identity for overshoots of a L\'{e}vy process. Then the value and the optimal stopping time for (\ref{e1.1}) can be expressed in terms of the solution to the averaging problem provided this solution has certain monotonicity properties. See also Deligiannidis, Le and Utev \cite{ref13} for related results on L\'{e}vy processes as well as on random walks. The work of Christensen, Salminen and Ta \cite{ref12} characterizes the solution of (\ref{e1.1}) similarly as in \cite{ref13,ref11} but under very general strong Markov processes including diffusions, L\'{e}vy processes and continuous-time Markov chains. Moreover, the solution can be either one-sided or two-sided depending on the representing function for the given reward function. More recently, Mordecki and Mishura \cite{ref7} have generalized Surya's averaging problem so as to make the construction method more flexible. As an example, an explicit solution is obtained for $g(x)=(x+\beta\sin x)\mathbf{1}_{(0,\infty)}(x)$ with $0<\beta<1$ under a compound Poisson process with a negative drift. Lately, Christensen \cite{ref15} has introduced an  auxiliary problem for a 2-dimensional process consisting of the underlying (Markov) process $\{X_t\}$ and its running maximum. Under suitable assumptions, the auxiliary problem turns out to have the infinitesimal look-ahead rule as its solution, which then yields an optimal stopping time for (\ref{e1.1}).

\section{Optimal stopping for random walks}
\hspace*{18pt}In this section, we use $n\in\mathbb{Z}^+$ (instead of $t$) to denote the discrete time parameter. Let $\xi,\xi_1,\xi_2,\dots$ be a sequence of real-valued independent and identically distributed (i.i.d.) random variables. To avoid trivial cases, assume that $P(\xi>0)>0$. For $X_0=x\in\mathbb{R}$, let $X_{n+1}=X_n+\xi_{n+1}$ for $n=0,1,\dots$, so that $\{X_n\}_{n\ge0}$ is a random walk with initial state $x$. For $y\in\mathbb{R}\cup\{-\infty\}$, define $\tau_y=\inf\{n\ge0:X_n\ge y\}$ (a threshold-type stopping time) and $T_y=\inf\{n\ge1:X_n\ge y\}$ (which is different from $\tau_y$ if $X_0\ge y$). Consider a nonnegative reward function $g:\mathbb{R}\to[0,\infty)$ which is nonconstant, increasing (i.e. $g(x)\le g(y)$ for $x\le y$) and logconcave (i.e. $g(\theta x+(1-\theta)y)\geq\left(g(x)\right)^{\theta}\left(g(y)\right)^{1-\theta}$ for all $x,y$ and $0<\theta<1$). Letting $\log 0:=-\infty$, the function $h(x):=\log g(x)$ is increasing and concave, so that the left-hand derivative $h'(x-)$ is well defined (possibly $+\infty$) at every $x$ with $h(x)>-\infty$. Letting $h'(x-):=+\infty$ if $h(x)=-\infty$, we have that $h'(x-)$ is decreasing (and nonnegative) in $x\in \mathbb{R}$.
Define
\begin{align}
&\beta=\lim_{x\to\infty}h'(x-),\label{e1}\\
&u=\inf\left\{x\in\mathbb{R}:\frac{E_x\left(e^{-qT_x}g\left(X_{T_x}\right)\mathbf{1}_{\left\{T_x<\infty\right\}}\right)}{g(x)}\le1\right\}\;\;\left(\frac{a}{0}:=\infty\;\;\mbox{for}\;\;a\ge0\right),\label{g1}\\
&W=\sup_{y\in\mathbb{R}}E_0(e^{-q\tau_y}g(X_{\tau_y})\mathbf{1}_{\left\{\tau_y<\infty\right\}})=\sup_{y\in\mathbb{R}}E(e^{-q\tau_y}g(X_{\tau_y})\mathbf{1}_{\left\{\tau_y<\infty\right\}}),\label{e2}\\
&V(x)=\sup_{\tau\in\mathcal{M}}E_{x}\left(e^{-q\tau}g(X_{\tau})\mathbf{1}_{\left\{\tau<\infty\right\}}\right)\;(x\in\mathbb{R}).\label{g2}
\end{align}
(While the subscript $x$ in $E_x$ refers to the initial state $X_0=x$, in case $x=0$, we write $E=E_0$ for simplicity as in (\ref{e2}).)

\begin{rem}
A nonnegative, nonconstant, increasing and logconcave function $g$ is continuous everywhere except possibly at
\begin{equation}\label{v1}
x_0:=\inf\{s\in\mathbb{R}:g(s)>0\}.
\end{equation}
Note that $-\infty\le x_0<\infty$. Moreover, since $g$ is nonconstant, $g$ is not identically 0 while $\lim_{x\to-\infty}g(x)=0$. In Theorem \ref{t2.1} below, $g$ is assumed to be right-continuous, which implies that $g(x_0)=g(x_0+)$ if $x_0>-\infty$.
\end{rem}

\begin{rem}
Note that $\mathcal{L}(T_x,X_{T_x}\mathbf{1}_{\{T_x<\infty\}}\mid X_0=x)=\mathcal{L}(T_0,(x+X_{T_0})\mathbf{1}_{\{T_0<\infty\}}\mid X_0=0)$, where $\mathcal{L}(Z)$ denotes the law of a random vector $Z$. Since by logconcavity, $g(x+\delta)/g(x)$ is decreasing in $x$ for $\delta\ge0$, it follows that
\begin{equation}\label{e2.18}
\frac{E_x\left(e^{-qT_x}g\left(X_{T_x}\right)\mathbf{1}_{\left\{T_x<\infty\right\}}\right)}{g(x)}=\frac{E\left(e^{-qT_0}g\left(x+X_{T_0}\right)\mathbf{1}_{\left\{T_0<\infty\right\}}\right)}{g(x)}\;\;\mbox{is decreasing in}\;\;x.
\end{equation}
It can be argued that if $G(x):=E_x\left(e^{-qT_x}g\left(X_{T_x}\right)\mathbf{1}_{\left\{T_x<\infty\right\}}\right)=\infty$ for some $x$, then $G(x)=\infty$ for all $x\in\mathbb{R}$, in which case we have $u=\infty$. Thus, if $u<\infty$, then $G(x)<\infty$ is continuous and increasing in $x>x_0$, so that
\begin{equation*}
\frac{E_x\left(e^{-qT_x}g\left(X_{T_x}\right)\mathbf{1}_{\left\{T_x<\infty\right\}}\right)}{g(x)}\;\left(=\frac{G(x)}{g(x)}\right)\;\mbox{is continuous}\;(\mbox{and decreasing})\;\mbox{in}\;\;x>x_0,
\end{equation*}
where $x_0$ is given in $(\ref{v1})$. It follows that
\begin{equation}\label{e2.17}
E_u\left(e^{-qT_u}g\left(X_{T_u}\right)\mathbf{1}_{\left\{T_u<\infty\right\}}\right)=g(u)\;\;\mbox{provided}\;\;x_0<u<\infty.
\end{equation}
\end{rem}

\begin{thm}\label{t2.1}
Let $g:\mathbb{R}\to[0,\infty)$ be nonconstant, increasing, logconcave and right-continuous, and define $\beta$, $u$, $W$ and $V(x)$ as in $(\ref{e1})$--$(\ref{g2})$. Assume $P(\xi>0)>0$. Then the following statements hold.
\begin{enumerate}
\item[\upshape(i)]If $-\infty\le u<\infty$, then the threshold-type stopping time $\tau_u$ is optimal, i.e. $V(x)=E_{x}\left(e^{-q\tau_u}g(X_{\tau_u})
\mathbf{1}_{\left\{\tau_u<\infty\right\}}\right)$ for all $x\in\mathbb{R}$.
\item[\upshape(ii)]If $u=\infty$, then $V(x)=e^{\beta x}W$ for all $x$. If, in addition, $W=\infty$, then there exist $($randomized$)$ stopping times that yield an infinite expected $($discounted$)$ reward; if $W<\infty$, then $V(x)=\lim_{y\to\infty}E_x\left(e^{-q\tau_{y}}g(X_{\tau_y})\mathbf{1}_{\{\tau_y<\infty\}}\right)$ but there is no optimal stopping time.
\end{enumerate}
\end{thm}

\begin{col}\label{c1}
\begin{itemize}
\item[\upshape(i)] $V(x)>g(x)$ for $x<u$ and $V(x)=g(x)$ for $x\ge u$.
\item[\upshape(ii)] $V(x)/g(x)$ is decreasing in $x$.
\item[\upshape(iii)] If $x_0<u<\infty$ or $u=-\infty$ or $u=\infty$ and $W<\infty$, then $V(x)$ is continuous everywhere.
\end{itemize}
\end{col}

\begin{rem}\label{r2.3}
By Corollary \ref{c1}, $u=\inf\{x:V(x)=g(x)\}=\sup\{x:V(x)>g(x)\}$. By Theorem \ref{t2.1}, $(-\infty,u)$ and $[u,\infty)$ are the $($optimal$)$ continuation and stopping regions, respectively. If $u=-\infty$, there is no continuation region, so that stopping immediately is optimal. If $u=\infty$, there is no stopping region, so that it is never optimal to stop since more profit can be made by stopping at a later time. However, to never stop yields a zero reward.
\end{rem}

\begin{rem}
Novikov and Shiryaev \cite{ref9} have solved $(\ref{e1.1})$ for $g(x)=(x^+)^{\nu}$ with $\nu>0$, and found that the optimal threshold is the positive root of the associated Appell function, which generalizes an earlier result of Darling et al. \cite{ref1} for $\nu=1$. It can be shown that their optimal threshold agrees with $(\ref{g1})$. As an illustration, consider $g(x)=x^+$ and $q=0$, for which Darling et al. \cite{ref1} showed that if $E(\xi)<0$, then $\tau_{E(M)}$ is optimal where $M=\sup\{0,\xi_1,\xi_1+\xi_2,\dots\}$. For $E(\xi)<0$, the value of $u$ defined in $(\ref{g1})$ satisfies $(\textit{cf.}\;(\ref{e2.17}))$
$$
0<u=g(u)=E_u\left(g(X_{T_u})\mathbf{1}_{\{T_u<\infty\}}\right)=E_u\left(X_{T_u}\mathbf{1}_{\{T_u<\infty\}}\right)=E\left[(u+X_{T_0})\mathbf{1}_{\{T_0<\infty\}}\right],
$$
yielding
$$
u=\frac{E\left(X_{T_0}\mathbf{1}_{\{T_0<\infty\}}\right)}{P(T_0=\infty)}=E(M),
$$
where the second equality follows from the fact that $\mathcal{L}(M)=\mathcal{L}\left((X_{T_0}+M')\mathbf{1}_{\{T_0<\infty\}}\mid X_0=0\right)$ with $M'$ being an independent copy of $M$. When $E(\xi)\ge0$ or $E(\xi)$ is undefined $(\mbox{i.e.}$\;$E(\max\{\xi,0\})$ $=\infty=E(\max\{-\xi,0\}))$, it is readily shown that $u=\infty=E(M)$. Thus, the value of $u$ defined in $(\ref{g1})$ equals $E(M)$ regardless of whether $E(\xi)<0$.
\end{rem}

\begin{rem}\label{r1}
For $x\in\mathbb{R}$ with $g(x)>0$, it can be shown that
$$
E_x\left(e^{-qT_x}g\left(X_{T_x}\right)\mathbf{1}_{\left\{T_x<\infty\right\}}\right)/g(x)\le1\;\;\;\mbox{if and only if}\;\;\;E_x\left(e^{-q\tau_{x+}}g\left(X_{\tau_{x+}}\right)\mathbf{1}_{\left\{\tau_{x+}<\infty\right\}}\right)/g(x)\le1,
$$
where $\tau_{x+}:=\inf\{n\ge0:X_n>x\}$. It follows that the definition of $u$ in $(\ref{g1})$ is equivalent to
$$
u=\inf\left\{x\in\mathbb{R}:E_x\left(e^{-q\tau_{x+}}g\left(X_{\tau_{x+}}\right)\mathbf{1}_{\left\{\tau_{x+}<\infty\right\}}\right)/g(x)\le1\right\}.
$$
If $x_0<u<\infty$, it follows from $(\ref{e2.17})$ that
$$
g(u)=E_u\left(e^{-qT_u}g(X_{T_u})\mathbf{1}_{\{T_u<\infty\}}\right)=E_u\left(e^{-q\tau_{u+}}g(X_{\tau_{u+}})\mathbf{1}_{\{\tau_{u+}<\infty\}}\right),
$$
which together with the optimality of $\tau_u$ implies that $\tau_{u+}$ is also optimal. However, $\tau_{u+}$ may not be optimal if $u=x_0>-\infty$. As an example, consider the $($logconcave$)$ function $g(x)=\mathbf{1}_{[0,\infty)}(x)$, for which $u=x_0=0$. While $\tau_0$ is optimal, we have $g(0)=1>E_0\left(e^{-q\tau_{0+}}g(X_{\tau_{0+}})\mathbf{1}_{\{\tau_{0+}<\infty\}}\right)$ if $q>0$ or $E(\xi)<0$.
\end{rem}

To prove Theorem \ref{t2.1}, we need the following lemmas.

\begin{lem}\label{lg}
For a nonnegative function $f$ defined on $\mathbb{R}$ and $v\in\mathbb{R}$, let
$$
U(x)=E_x(e^{-q\tau_v}f(X_{\tau_v})\mathbf{1}_{\{\tau_v<\infty\}}).
$$
Then $E(e^{-q}U(x+\xi))=E_x(e^{-qT_v}f(X_{T_v})\mathbf{1}_{\{T_v<\infty\}})$.
\end{lem}

\begin{lem}\label{lb}
Let $f(x)$ and $g(x)$ be nonnegative  functions defined on $\mathbb{R}$. If $f(x)\geq g(x)$ and $f(x)\geq E[e^{-q}f(x+\xi)]$ for all $x$, then
$$
f(x)\geq \sup_{\tau\in\mathcal{M}}E_{x}\left(e^{-q\tau}g(X_{\tau})\mathbf{1}_{\left\{\tau<\infty\right\}}\right), \;\; x\in\mathbb{R}.
$$
\end{lem}

\begin{lem}\label{lc}
Assume $P(\xi>0)>0$. Let $g:\mathbb{R}\to[0,\infty)$ be nonconstant, increasing and logconcave. Suppose that
\begin{equation}\label{e2.3}
\frac{E_y\left(e^{-qT_y}g(X_{T_y})\mathbf{1}_{\left\{T_y<\infty\right\}}\right)}{g(y)}\ge1\;\;\mbox{for some}\;\;y>x_0:=\inf\{s:g(s)>0\}.
\end{equation}
Then the following hold.
\begin{enumerate}
\item[\upshape(i)] $E_x\left(e^{-qT_a}g(X_{T_a})\mathbf{1}_{\left\{T_a<\infty\right\}}\right)\le E_x\left(e^{-qT_y}g(X_{T_y})\mathbf{1}_{\left\{T_y<\infty\right\}}\right)$ for $x\in\mathbb{R}$ and $a\in[-\infty,y)$.
\item[\upshape(ii)] $g(x)\le E_x\left(e^{-q\tau_y}g(X_{\tau_y})\mathbf{1}_{\left\{\tau_y<\infty\right\}}\right)$ for $x\in(-\infty,y)$.
\end{enumerate}
\end{lem}

\begin{lem}\label{ld}
Assume $P(\xi>0)>0$. Let $g:\mathbb{R}\to[0,\infty)$ be nonconstant, increasing and logconcave. Suppose that
\begin{equation}\label{e9}
\frac{E_y\left(e^{-qT_y}g(X_{T_y})\mathbf{1}_{\left\{T_y<\infty\right\}}\right)}{g(y)}\le1\;\;\mbox{for some}\;\;y>x_0:=\inf\{s:g(s)>0\}.
\end{equation}
Then the following hold.
\begin{enumerate}
\item[\upshape(i)] $E_x\left(e^{-qT_y}g(X_{T_y})\mathbf{1}_{\left\{T_y<\infty\right\}}\right)\ge E_x\left(e^{-qT_a}g(X_{T_a})\mathbf{1}_{\left\{T_a<\infty\right\}}\right)$ for $x\in\mathbb{R}$ and $a\in(y,\infty)$.
\item[\upshape(ii)] $g(y)\ge E_y\left(e^{-q\tau_a}g(X_{\tau_a})\mathbf{1}_{\left\{\tau_a<\infty\right\}}\right)$ for $a\in(y,\infty)$.
\end{enumerate}
\end{lem}

\begin{lem}\label{la}
Assume $P(\xi>0)>0$ and $\phi(\lambda):=E(e^{\lambda\xi})<\infty$ for all $\lambda\ge0$. Let $g:\mathbb{R}\to[0,\infty)$ be nonconstant, increasing, logconcave and right-continuous. Define $u$ as in $(\ref{g1})$. Suppose $u<\infty$. Then $\tau_u$ is optimal. Moreover, the value function $V(x)$ satisfies $V(x)>g(x)$ for $x<u$ and $V(x)=g(x)$ for $x\ge u$.
\end{lem}

\begin{lem}\label{le}
Assume $P(\xi>0)>0$. Let $g:\mathbb{R}\to[0,\infty)$ be nonconstant, increasing, logconcave and right-continuous. Define $u$ as in $(\ref{g1})$. Suppose there exist $x'\in\mathbb{R}$ and $c>0$ such that $g(x)=c$ for $x\ge x'$. Then $u\le x'$ and $\tau_u$ is optimal. Moreover, the value function $V(x)$ satisfies $V(x)>g(x)$ for $x<u$ and $V(x)=g(x)$ for $x\ge u$.
\end{lem}

Lemma \ref{lg} follows easily by conditioning on $X_1=x+\xi$. Lemma \ref{lb} is a standard result (see  \cite[Lemma 5]{ref8}). The proofs of Lemmas \ref{lc}--\ref{le} are relegated to Section 6. We are now ready to prove Theorem \ref{t2.1}.

\begin{proof}[\bf Proof of Theorem \ref{t2.1}(i)]
Let $\{b_k\}_{k\ge1}$ be an increasing sequence such that $b_1>x_0:=\inf\{s:g(s)>0\}$ and $\lim_{k\to\infty}b_k=\infty$. For each $k\ge1$, let $g_k(x)=g(x\wedge b_k)$ for $x\in\mathbb{R}$ where $x\wedge b_k:=\min\{x,b_k\}$, and let
\begin{equation}\label{e13}
V_k(x)=\sup_{\tau\in\mathcal{M}}E_x(e^{-q\tau}g_k(X_\tau)\mathbf{1}_{\{\tau<\infty\}}).
\end{equation}
Then for $k\ge1$, $g_k(x)$ is nonconstant, increasing, logconcave and right-continuous. Since $g_k(x)$ is increasing in $k$, so is $V_k(x)$. Note that $g_k(x)=g(b_k)>0$ for $x\ge b_k$ and $g_k(x)=g(x)$ for $x<b_k$. We have that $V_k(x)\le g(b_k)$ for $x\in\mathbb{R}$ and by Lemma \ref{le} that $\tau_{u_k}$ is optimal for the optimal stopping problem (\ref{e13}) with reward function $g_k$, where
\begin{equation}\label{v2}
u_k:=\inf\left\{x\in\mathbb{R}:\frac{E_x(e^{-qT_x}g_k(X_{T_x})\mathbf{1}_{\{T_x<\infty\}})}{g_k(x)}\le1\right\}\le b_k<\infty.
\end{equation}
We claim that $u_k$ is increasing in $k$. If $u_{k+1}\ge b_k$, then $u_{k+1}\ge u_k$ clearly. In case $u_{k+1}<b_k$, we have $V_k(u_{k+1})\le V_{k+1}(u_{k+1})=g_{k+1}(u_{k+1})=g(u_{k+1})=g_k(u_{k+1})$, implying that $u_{k+1}\ge u_k$.

Let $u_\infty=\lim_{k\to\infty}u_k$ and $V_\infty(x)=\lim_{k\to\infty}V_k(x)$. Then it is easily seen that $V(x)=V_\infty(x)$. Now we prove in three steps that $\tau_u$ is optimal if $-\infty\le u<\infty$. We show in step 1 that $u_\infty\le u\;(<\infty)$, in step 2 that $\tau_{u_\infty}$ is optimal, and in step 3 that $u_\infty\ge u$.

{\it Step 1.} To prove $u_\infty\le u$, suppose to the contrary that $u<u_\infty$. Choose an $x$ and a (large) $k$ such that $u<x<u_k$ and $b_k\ge x$. We have
\begin{align*}
g_k(x)&=g(x)\\
&\ge E_x\left(e^{-qT_x}g\left(X_{T_x}\right)\mathbf{1}_{\left\{T_x<\infty\right\}}\right)\;\;\mbox{(by (\ref{g1}))}\\
&\ge E_x\left(e^{-qT_x}g_k\left(X_{T_x}\right)\mathbf{1}_{\left\{T_x<\infty\right\}}\right),
\end{align*}
which together with (\ref{v2}) implies that $x\ge u_k$, a contradiction. This proves that $u_\infty\le u$.

{\it Step 2.} To prove that $\tau_{u_\infty}$ is optimal, it suffices to show that
\begin{equation}\label{e14}
E_x\left(e^{-q\tau_{u_\infty}}g(X_{\tau_{u_\infty}})\mathbf{1}_{\{\tau_{u_\infty}<\infty\}}\right)\ge\sup_{k}V_k(x)\;(=V_\infty(x)=V(x))\;\;\mbox{for all}\;\;x.
\end{equation}
If $u_\infty=-\infty$, then $u_\infty=u_k=-\infty$ for all $k$, implying, for all $x$, that
$$
V_k(x)=g_k(x)\le g(x)=E_x\left(e^{-q\tau_{u_\infty}}g(X_{\tau_{u_\infty}})\mathbf{1}_{\{\tau_{u_\infty}<\infty\}}\right)\;\;\mbox{for all}\;\;k,
$$
establishing (\ref{e14}) for the case $u_\infty=-\infty$.

Suppose $-\infty<u_\infty\;(\le u<\infty)$. Since $u_\infty\ge u_k$ for all $k$, we have for $x\ge u_\infty$,
$$
E_x\left(e^{-q\tau_{u_\infty}}g(X_{\tau_{u_\infty}})\mathbf{1}_{\{\tau_{u_\infty}<\infty\}}\right)=g(x)\ge g_k(x)=V_k(x)\;\;\mbox{for all}\;\;k.
$$
It remains to prove (\ref{e14}) for $x<u_\infty$. Note that $u_\infty\ge u_k\ge x_0$ for all $k$. If $u_\infty=x_0$, then $u_k=u_\infty=x_0$ for all $k$, so that
$$
V_k(x)=E_x\left(e^{-q\tau_{u_\infty}}g_k(X_{\tau_{u_\infty}})\mathbf{1}_{\{\tau_{u_\infty}<\infty\}}\right)\le E_x\left(e^{-q\tau_{u_\infty}}g(X_{\tau_{u_\infty}})\mathbf{1}_{\{\tau_{u_\infty}<\infty\}}\right)\;\;\mbox{for all}\;\;k,
$$
proving (\ref{e14}).

Now suppose that $u_\infty>x_0$. To prove (\ref{e14}) for $x<u_\infty$, let $k_0$ be so large that $x<u_{k_0}$, $2u_{k_0}-u_\infty>x_0$ and $b_{k_0}>u_\infty$. Thus for all $k\ge k_0$, we have $x<u_k$, $2u_k-u_\infty>x_0$ and $b_k>u_\infty$. For $k\ge k_0$, let
\begin{align*}
U_k(x)&:=E_x\left(e^{-q\tau_{u_\infty}}g_k(X_{\tau_{u_\infty}})\mathbf{1}_{\{\tau_{u_\infty}<\infty\}}\right),\\ \varepsilon_k&:=u_\infty-u_k\;\;(\ge0),\\
\tau'=\tau'(k)&:=\inf\{n\ge0:X_n\ge u_k-\varepsilon_k\},\\ \mbox{and}\;\;V'_k(y)&:=E_y\left(e^{-q\tau'}g_k(X_{\tau'})\mathbf{1}_{\{\tau'<\infty\}}\right)\;\;(\le V_k(y))\;\;\mbox{for}\;\;y\in\mathbb{R}. \end{align*}
Since $\mathcal{L}(\tau_{u_\infty},X_{\tau_{u_\infty}}\mathbf{1}_{\{\tau_{u_\infty}<\infty\}}\mid X_0=x)=\mathcal{L}(\tau_{u_k},(\varepsilon_k+X_{\tau_{u_k}})\mathbf{1}_{\{\tau_{u_k}<\infty\}}\mid X_0=x-\varepsilon_k)$, we have
\begin{align}
U_k(x)&=E_{x}\left(e^{-q\tau_{u_\infty}}g_k(X_{\tau_{u_\infty}})\mathbf{1}_{\{\tau_{u_\infty}<\infty\}}\right)\notag\\
&=E_{x-\varepsilon_k}\left(e^{-q\tau_{u_k}}g_k(\varepsilon_k+X_{\tau_{u_k}})\mathbf{1}_{\{\tau_{u_k}<\infty\}}\right)\notag\\
&\ge E_{x-\varepsilon_k}\left(e^{-q\tau_{u_k}}g_k(X_{\tau_{u_k}})\mathbf{1}_{\{\tau_{u_k}<\infty\}}\right)\notag\\
&=V_k(x-\varepsilon_k)\ge V'_k(x-\varepsilon_k).\label{e2.13}
\end{align}
Define $\tau''=\inf\{n\ge0:X_n\ge u_k-x\}$. On $\{\tau''<\infty\}$,
$$
g_k(x-\varepsilon_k+X_{\tau''})\ge g_k(u_k-\varepsilon_k)=g_k(2u_k-u_\infty)=g((2u_k-u_\infty)\wedge b_k)=g(2u_k-u_\infty)>0.
$$
Since $\mathcal{L}(\tau_{u_k},X_{\tau_{u_k}}\mathbf{1}_{\{\tau_{u_k}<\infty\}}\mid X_0=x)=\mathcal{L}(\tau'',(x+X_{\tau''})\mathbf{1}_{\{\tau''<\infty\}}\mid X_0=0)$, we have
\begin{align}
V_k(x)&=E_x\left(e^{-q\tau_{u_k}}g_k(X_{\tau_{u_k}})\mathbf{1}_{\{\tau_{u_k}<\infty\}}\right)\notag\\
&=E\left(e^{-q\tau''}g_k(x+X_{\tau''})\mathbf{1}_{\{\tau''<\infty\}}\right)\notag\\
&=E\left(e^{-q\tau''}\frac{g_k(x+X_{\tau''})}{g_k(x-\varepsilon_k+X_{\tau''})}g_k(x-\varepsilon_k+X_{\tau''})\mathbf{1}_{\{\tau''<\infty\}}\right)\notag\\
&\le\frac{g_k(u_k)}{g_k(2u_k-u_\infty)}E\left(e^{-q\tau''}g_k(x-\varepsilon_k+X_{\tau''})\mathbf{1}_{\{\tau''<\infty\}}\right)\notag\\
&=\frac{g_k(u_k)}{g_k(2u_k-u_\infty)}E_{x-\varepsilon_k}\left(e^{-q\tau'}g_k(X_{\tau'})\mathbf{1}_{\{\tau'<\infty\}}\right)\notag\\
&=\frac{g_k(u_k)}{g_k(2u_k-u_\infty)}V'_k(x-\varepsilon_k),\label{e2.14}
\end{align}
where the inequality follows from the logconcavity of $g_k$ (noting that $X_{\tau''}\ge u_k-x$ on $\{\tau''<\infty\}$) and the second-to-last equality is due to the fact that
$$
\mathcal{L}(\tau'',(x-\varepsilon_k+X_{\tau''})\mathbf{1}_{\{\tau''<\infty\}}\mid X_0=0)=\mathcal{L}(\tau',X_{\tau'}\mathbf{1}_{\{\tau'<\infty\}}\mid X_0=x-\varepsilon_k).
$$
By (\ref{e2.13}) and (\ref{e2.14}), we have (for $k\ge k_0$)
\begin{equation}\label{v3}
U_k(x)\ge\frac{g_k(2u_k-u_\infty)}{g_k(u_k)}V_k(x)=\frac{g(2u_k-u_\infty)}{g(u_k)}V_k(x).
\end{equation}
Letting $k\to\infty$, the right- and left-hand sides of (\ref{v3}) tend, respectively, to $V_\infty(x)$ and
$E_x\left(e^{-q\tau_{u_\infty}}g(X_{\tau_{u_\infty}})\mathbf{1}_{\{\tau_{u_\infty}<\infty\}}\right)$, yielding (\ref{e14}).

{\it Step 3.} We now prove that $u_\infty\ge u$. Suppose to the contrary that $u>u_\infty\;(\ge x_0)$. Then it follows from (\ref{g1}) that
$$
E_{u_\infty}\left(e^{-qT_{u_\infty}}g\left(X_{T_{u_\infty}}\right)\mathbf{1}_{\left\{T_{u_\infty}<\infty\right\}}\right)/g(u_\infty)>1,
$$
which implies by the optimality of $\tau_{u_\infty}$ (established in step 2) that
$$
g(u_\infty)<E_{u_\infty}\left(e^{-qT_{u_\infty}}g\left(X_{T_{u_\infty}}\right)\mathbf{1}_{\left\{T_{u_\infty}<\infty\right\}}\right)\le V(u_\infty)=g(u_\infty),
$$
a contradiction. Thus, $u\le u_\infty$. The proof is complete.
\end{proof}

\begin{proof}[\bf Proof of Theorem \ref{t2.1}(ii)]
For $x,y\in\mathbb{R}$, let
$$
Q_y(x):=E_x\left(e^{-q\tau_y}g(X_{\tau_y})\mathbf{1}_{\{\tau_y<\infty\}}\right)\;\;\mbox{and}\;\;Q(x):=\sup_{y\in\mathbb{R}}Q_y(x)\le V(x).
$$
Note by (\ref{e2}) that $Q(0)=W$. Since $u=\infty$, we have by (\ref{g1})
\begin{equation}\label{v4}
E_y\left(e^{-qT_y}g(X_{T_y})\mathbf{1}_{\{T_y<\infty\}}\right)>g(y)\;\;\mbox{for all}\;\;y>x_0.
\end{equation}
It follows from Lemma \ref{lc} that for $x,y_1,y_2\in\mathbb{R}$ with $y_1<y_2$ and $x_0<y_2$,
$$
E_{x}\left(e^{-qT_{y_1}}g(X_{T_{y_1}})\mathbf{1}_{\{T_{y_1}<\infty\}}\right)\le E_{x}\left(e^{-qT_{y_2}}g(X_{T_{y_2}})\mathbf{1}_{\{T_{y_2}<\infty\}}\right)
$$
and that for $y>\max\{x,x_0\}$,
$$
g(x)\le E_{x}\left(e^{-q\tau_{y}}g(X_{\tau_{y}})\mathbf{1}_{\{\tau_{y}<\infty\}}\right).
$$
In view of $Q_y(x)=g(x)$ for $x\ge y$ and $\tau_y=T_y$ for $X_0=x<y$, we have for $x,y,y'\in\mathbb{R}$ with $y>\max\{x,x_0,y'\}$ that
$$
Q_y(x)\ge\max\left\{g(x),E_{x}\left(e^{-q\tau_{y'}}g(X_{\tau_{y'}})\mathbf{1}_{\{\tau_{y'}<\infty\}}\right)\right\}\ge Q_{y'}(x),
$$
which implies that for $x\in\mathbb{R}$, $\sup_{y'\le x_0}Q_{y'}(x)\le Q_y(x)$ for $y>x_0$ and that
\begin{equation}\label{v5}
Q_y(x)\;\;\mbox{is increasing in}\;\;y\in(x_0,\infty).
\end{equation}
Consequently,
\begin{equation}\label{v6}
Q(x)=\sup_{y}Q_y(x)=\lim_{y\to\infty}Q_y(x).
\end{equation}
Since $\mathcal{L}(\tau_y,X_{\tau_y}\mathbf{1}_{\{\tau_y<\infty\}}\mid X_0=x)=\mathcal{L}(\tau_{y-x},(x+X_{\tau_{y-x}})\mathbf{1}_{\{\tau_{y-x}<\infty\}}\mid X_0=0)$, we have by (\ref{v6})
\begin{align}
Q(x)&=\lim_{y\to\infty}E_{x}\left(e^{-q\tau_{y}}g(X_{\tau_{y}})\mathbf{1}_{\{\tau_{y}<\infty\}}\right)\notag\\
&=\lim_{y\to\infty}E\left(e^{-q\tau_{y-x}}g(x+X_{\tau_{y-x}})\mathbf{1}_{\{\tau_{y-x}<\infty\}}\right)\notag\\
&=\lim_{y\to\infty}E\left(e^{-q\tau_{y}}g(x+X_{\tau_{y}})\mathbf{1}_{\{\tau_{y}<\infty\}}\right)\notag\\
&=\lim_{y\to\infty}E\left[\left(\frac{g(x+X_{\tau_y})}{g(X_{\tau_y})}\right)e^{-q\tau_{y}}g(X_{\tau_{y}})\mathbf{1}_{\{\tau_{y}<\infty\}}\right]\notag\\
&=e^{\beta x}\lim_{y\to\infty}E\left(e^{-q\tau_{y}}g(X_{\tau_{y}})\mathbf{1}_{\{\tau_{y}<\infty\}}\right)\notag\\
&=e^{\beta x}Q(0)=e^{\beta x}W,\label{v7}
\end{align}
where we have used the facts that on $\{\tau_y<\infty\}$, $X_{\tau_y}\ge y$ and
$$
\inf_{z\ge y}\frac{g(x+z)}{g(z)}\le\frac{g(x+X_{\tau_y})}{g(X_{\tau_y})}\le\sup_{z\ge y}\frac{g(x+z)}{g(z)}.
$$
Note that
\begin{align*}
\inf_{z\ge y}\frac{g(x+z)}{g(z)}&=\inf_{z\ge y}e^{h(x+z)-h(z)}\\
&\ge\inf_{z\ge y}e^{h'((x+z)-)x}\\
&=\min\left\{e^{h'((x+y)-)x},e^{\beta x}\right\}\to e^{\beta x}\;\;(\mbox{as}\;\;y\to\infty)
\end{align*}
and
\begin{align*}
\sup_{z\ge y}\frac{g(x+z)}{g(z)}&=\sup_{z\ge y}e^{h(x+z)-h(z)}\\
&\le\sup_{z\ge y}e^{h'(z-)x}\\
&=\max\left\{e^{h'(y-)x},e^{\beta x}\right\}\to e^{\beta x}\;\;\mbox{as}\;\;y\to\infty.
\end{align*}
Suppose $W=\infty$. Then $V(x)\ge Q(x)=e^{\beta x}W=\infty$, implying that $V(x)=\infty=e^{\beta x}W$. In view of $Q_y(x)\to\infty$ as $y\to\infty$, let $y_k$, $k=1,2,\dots$ be such that
$$
Q_{y_k}(x)=E_{x}\left(e^{-q\tau_{y_k}}g(X_{\tau_{y_k}})\mathbf{1}_{\{\tau_{y_k}<\infty\}}\right)>2^k.
$$
Let $\tau$ be a randomized stopping time of threshold type which chooses the threshold $y_k$ with probability $2^{-k}$, $k=1,2,\dots$. Then we have $E_{x}\left(e^{-q\tau}g(X_{\tau})\mathbf{1}_{\{\tau<\infty\}}\right)=\infty$.

Now suppose $W<\infty$. We have by (\ref{v7}) that
$$
e^{\beta x}W=Q(x)=\sup_{y}E_{x}\left(e^{-q\tau_{y}}g(X_{\tau_{y}})\mathbf{1}_{\{\tau_{y}<\infty\}}\right)\ge g(x).
$$
To prove $V(x)=Q(x)$, by Lemma \ref{lb} it suffices to show that
$$
Q(x)\ge E(e^{-q}Q(x+\xi)).
$$
Indeed, we claim that the equality holds, i.e.
\begin{equation}\label{v8}
Q(x)=E(e^{-q}Q(x+\xi)).
\end{equation}
By Lemma \ref{lg}, for $y>x$
$$
Q_y(x)=E_x\left(e^{-q\tau_y}g(X_{\tau_y})\mathbf{1}_{\{\tau_y<\infty\}}\right)=E_x\left(e^{-qT_y}g(X_{T_y})\mathbf{1}_{\{T_y<\infty\}}\right)=E(e^{-q}Q_y(x+\xi)).
$$
By (\ref{v5}), $Q_y(\cdot)$ is increasing in $y\in(x_0,\infty)$. It follows from the monotone convergence theorem and (\ref{v6}) that
\begin{align*}
Q(x)&=\lim_{y\to\infty}Q_y(x)\\
&=\lim_{y\to\infty}E(e^{-q}Q_y(x+\xi))\\
&=E(e^{-q}Q(x+\xi)),
\end{align*}
proving (\ref{v8}). Hence, $V(x)=Q(x)=e^{\beta x}W$.

Finally, we show that for any stopping time $\tau$,
\begin{equation}\label{v9}
V(x)>E_x(e^{-q\tau}g(X_\tau)\mathbf{1}_{\{\tau<\infty\}}),
\end{equation}
which implies that there exists no optimal stopping time. By (\ref{v8}), $\{e^{-qn}V(X_n)\}_{n\ge0}=\{e^{-qn}Q(X_n)\}_{n\ge0}$ is a (positive) martingale. To prove (\ref{v9}) for any stopping time $\tau$, it suffices to assume that $\tau$ is finite with a positive probability. Then
\begin{align*}
V(x)&\ge E_x(e^{-q\tau}V(X_\tau)\mathbf{1}_{\{\tau<\infty\}})\\
&>E_x(e^{-q\tau}g(X_\tau)\mathbf{1}_{\{\tau<\infty\}}),
\end{align*}
where the strict inequality follows from the fact that $V(y)>g(y)$ for all $y$. The proof is complete.
\end{proof}

\begin{rem}
Combining $(\ref{v7})$, $(\ref{v8})$ and $0<W<\infty$ yields $E(e^{\beta\xi})=e^{q}$. This is a consequence of $u=\infty$ and $W<\infty$.
\end{rem}

\begin{proof}[\bf Proof of Corollary \ref{c1}]
(i) If $u=\infty$, it holds trivially that $V(x)=g(x)$ for $x\ge u$. (It is in fact vacuous.) If $u<\infty$, we have by Theorem \ref{t2.1}(i) that $V(x)=g(x)$ for $x\ge u$. To show $V(x)>g(x)$ for $x<u$, note that for any $y$ with $g(y)>0$,
$$
V(x)\ge E_x\left(e^{-q\tau_y}g(X_{\tau_y})\mathbf{1}_{\{\tau_y<\infty\}}\right)>0,
$$
since $P_x(\tau_y<\infty)>0$ by the assumption of $P(\xi>0)>0$. Thus, $V(x)>g(x)$ for all $x$ with $g(x)=0$. For $x<u$ with $g(x)>0$, we have by (\ref{g1}) that
$$
V(x)\ge E_x\left(e^{-qT_x}g(X_{T_x})\mathbf{1}_{\{T_x<\infty\}}\right)>g(x).
$$

(ii) If $u=-\infty$, $V(x)/g(x)=1$ for all $x$. If $u=\infty$ and $W=\infty$, $V(x)/g(x)=\infty$ for all $x$. If $u=\infty$ and $W<\infty$, $V(x)/g(x)=e^{\beta x}W/g(x)$. For $x>x_0$, $e^{\beta x}/g(x)=e^{\beta x-h(x)}$ which is decreasing in $x>x_0$ since $\beta-h'(x-)\le0$. It remains to deal with the case $-\infty<u<\infty$. Since $V(x)/g(x)=\infty$ for $x$ satisfying $g(x)=0$ and $V(x)/g(x)=1$ for $x\ge u$, it suffices to show
$$
\frac{V(x)}{g(x)}\ge\frac{V(y)}{g(y)}\;\;\mbox{for}\;\;x<y\le u\;\;\mbox{with}\;\;g(x)>0.
$$
Since
\begin{equation}\label{c8}
\mathcal{L}(\tau_u,X_{\tau_u}\mathbf{1}_{\{\tau_u<\infty\}}\mid X_0=y)=\mathcal{L}(\tau_{x-y+u},(y-x+X_{\tau_{x-y+u}})\mathbf{1}_{\{\tau_{x-y+u}<\infty\}}\mid X_0=x),
\end{equation}
we have
\begin{align*}
\frac{V(y)}{g(y)}&=\frac{E_y\left(e^{-q\tau_u}g(X_{\tau_u})\mathbf{1}_{\{\tau_u<\infty\}}\right)}{g(y)}\\
&=\frac{E_x\left(e^{-q\tau_{x-y+u}}g(y-x+X_{\tau_{x-y+u}})\mathbf{1}_{\{\tau_{x-y+u}<\infty\}}\right)}{g(y-x+x)}\\
&\le\frac{E_x\left(e^{-q\tau_{x-y+u}}g(X_{\tau_{x-y+u}})\mathbf{1}_{\{\tau_{x-y+u}<\infty\}}\right)}{g(x)}\\
&\le\frac{V(x)}{g(x)},
\end{align*}
where the first inequality follows from the fact that by the logconcavity of $g$,
$$
\frac{g(y-x+X_{\tau_{x-y+u}})}{g(y-x+x)}\le\frac{g(X_{\tau_{x-y+u}})}{g(x)}\;\;\mbox{on}\;\;\{\tau_{x-y+u}<\infty\}.
$$

(iii) Note that $g(x)$ is continuous everywhere except possibly at $x=x_0$. If $u=-\infty$, then $V(x)=g(x)>0$ for all $x$, so that $V(x)$ is continuous everywhere. If $u=\infty$ and $W<\infty$, then $V(x)=e^{\beta x}W$, which is continuous everywhere. Suppose $-\infty<u<\infty$ and $u>x_0$. Since $V(x)=g(x)$ for $x\ge u$, $V(x)$ is continuous at $x\in(u,\infty)$. Moreover, $V(u+)=g(u+)=g(u)=V(u)$, and
$$
\lim_{x\uparrow u}V(x)\ge\lim_{x\uparrow u}g(x)=g(u-)=g(u)=V(u),
$$
showing by the monotonicity of $V(x)$ that $V(x)$ is continuous at $x=u$. To show that $V(x)$ is continuous at $x\in(-\infty,u)$, consider $x<y<u$ with $x-y+u>x_0$. We have by (\ref{c8})
\begin{align*}
\frac{V(y)}{g(u)}&=\frac{E_y\left(e^{-q\tau_u}g(X_{\tau_u})\mathbf{1}_{\{\tau_u<\infty\}}\right)}{g(u)}\\
&=\frac{E_x\left(e^{-q\tau_{x-y+u}}g(y-x+X_{\tau_{x-y+u}})\mathbf{1}_{\{\tau_{x-y+u}<\infty\}}\right)}{g(y-x+x-y+u)}\\
&\le\frac{E_x\left(e^{-q\tau_{x-y+u}}g(X_{\tau_{x-y+u}})\mathbf{1}_{\{\tau_{x-y+u}<\infty\}}\right)}{g(x-y+u)}\;\;\mbox{(by the logconcavity of}\;g)\\
&\le\frac{V(x)}{g(x-y+u)}.
\end{align*}
So for $x<y<u$ with $y-x<u-x_0$,
$$
0\le V(y)-V(x)\le\left(\frac{g(u)}{g(u+x-y)}-1\right)V(x)\le\left(\frac{g(u)}{g(u+x-y)}-1\right)g(u),
$$
which together with $g(u-)=g(u)>0$ implies that $V(x)$ is (uniformly) continuous in $x\in(-\infty,u)$. The proof is complete.
\end{proof}

\section{Optimal stopping for L\'{e}vy processes}
\hspace*{18pt}Let $X=\{X_t\}_{t\ge0}$ be a L\'{e}vy process with initial state $X_0=x\in\mathbb{R}$. For a comprehensive discussion of L\'{e}vy processes, see \cite{ref0}, \cite{ref5} and \cite{ref10}.  As in Section 2, assume $P(X_1>0)>0$ since the case $P(X_1\le0)=1$ is trivial. For $q\ge0$ and  $g:\mathbb{R}\to[0,\infty)$, let
\begin{equation}\label{e3.1}
V(x)=\sup_{\tau\in\mathcal{M}}E_x\left(e^{-q\tau}g(X_\tau)\mathbf{1}_{\{\tau<\infty\}}\right),
\end{equation}
where $\mathcal{M}$ is the class of all stopping times $\tau$ taking values in $[0,\infty]$ with respect to the filtration $\{\mathcal{F}_t\}_{t\ge0}$, $\mathcal{F}_t$ being the natural enlargement of $\sigma\{X_s,0\le s\le t\}$.

To apply Theorem \ref{t2.1} to problem (\ref{e3.1}), we introduce a sequence of optimal stopping problems in discrete time (\textit{cf.} \cite[Chapter 3]{ref17}). For $\ell=1,2,\dots$, let $\mathcal{M}^{(\ell)}$ denote the class of all stopping times in $\mathcal{M}$ taking values in $\{n2^{-\ell}:n=0,1,\dots,\infty\}$ and
\begin{equation}\label{e3.2}
V^{(\ell)}(x)=\sup_{\tau\in\mathcal{M}^{(\ell)}}E_x\left(e^{-q\tau}g(X_\tau)\mathbf{1}_{\{\tau<\infty\}}\right),\;x\in\mathbb{R}.
\end{equation}
Note that by the Markov property of $X$, $V^{(\ell)}(x)$ equals the supremum of $E_x\left(e^{-q\tau}g(X_\tau)\mathbf{1}_{\{\tau<\infty\}}\right)$ over the (smaller) class of all stopping times $\tau$ taking values in $\{n2^{-\ell}:n=0,1,\dots,\infty\}$ such that $\{\tau=n2^{-\ell}\}\in\sigma\{X_{i2^{-\ell}},i=0,1,\dots,n\}\subset\mathcal{F}_{n2^{-\ell}}$. So we can apply Theorem \ref{t2.1} to problem (\ref{e3.2}).

Let $u^{(\ell)}$ be defined as in (\ref{g1}) with $T_x$ replaced by
$$
T_x^{(\ell)}:=2^{-\ell}\inf\{n\ge1:X_{n{2}^{-\ell}}\ge x\},
$$
i.e.,
\begin{equation}\label{b1}
u^{(\ell)}:=\inf\left\{x\in\mathbb{R}:\frac{E_x\left(e^{-qT_x^{(\ell)}}g(X_{T_x^{(\ell)}})\mathbf{1}_{\{T_x^{(\ell)}<\infty\}}\right)}{g(x)}\le1\right\}.
\end{equation}
For $y\in\mathbb{R}\cup\{-\infty\}$, let
$$
\tau_y:=\inf\{t\ge0:X_t\ge y\}\in\mathcal{M}
$$
and
$$
\tau_y^{(\ell)}=\tau^{(\ell)}(y):=2^{-\ell}\inf\{n\ge0:X_{n2^{-\ell}}\ge y\}\in\mathcal{M}^{(\ell)}.
$$
By Theorem \ref{t2.1}, if $u^{(\ell)}<\infty$, then $\tau_{u^{(\ell)}}^{(\ell)}=\tau^{(\ell)}(u^{(\ell)})$ is optimal for (\ref{e3.2}), i.e.
\begin{equation}\label{e3.3}
V^{(\ell)}(x)=E_x\left(e^{-q\tau^{(\ell)}(u^{(\ell)})}g\left(X_{\tau^{(\ell)}(u^{(\ell)})}\right)\mathbf{1}_{\{\tau^{(\ell)}(u^{(\ell)})<\infty\}}\right),\;x\in\mathbb{R}.
\end{equation}
By Lemma \ref{l3.1} below, $u^{(\ell)}$ is increasing in $\ell$. Let
\begin{equation}\label{b2}
u=\lim_{\ell\to\infty}u^{(\ell)}\;\;\mbox{and}\;\;W=\sup_{y\in\mathbb{R}}E\left(e^{-q\tau_y}g(X_{\tau_y})\mathbf{1}_{\{\tau_y<\infty\}}\right).
\end{equation}

\begin{rem}
In Section 2, we used the notations $V(x)$, $u$ and $W$ for the random walk setting. In this section, the same notations are used for the L\'{e}vy process setting. Moreover, we refer to the setting of random walk $\{X_{n2^{-\ell}},n=0,1,\dots\}$ by attaching the superscript $(\ell)$, e.g. $V^{(\ell)}(x)$, $u^{(\ell)}$, $\tau_x^{(\ell)}$.
\end{rem}

\begin{thm}\label{t3.1}
Let $q\ge0$ and $g:\mathbb{R}\to[0,\infty)$ be nonconstant, increasing, logconcave and right-continuous. Let $\beta=\lim_{x\to\infty}h'(x-)$ where $h(x)=\log g(x)$. Define $V(x)$, $u$ and $W$ as in $(\ref{e3.1})$ and $(\ref{b2})$.
\begin{itemize}
\item[\upshape(i)] If $-\infty\le u<\infty$, then $\tau_{u}$ is optimal for $(\ref{e3.1})$, i.e. $V(x)=E_x\left(e^{-q\tau_{u}}g(X_{\tau_{u}})\mathbf{1}_{\{\tau_{u}<\infty\}}\right)$ for all $x\in\mathbb{R}$.
\item[\upshape(ii)] If $u=\infty$, then $V(x)=e^{\beta x}W$ for $x\in\mathbb{R}$. If, in addition, $W=\infty$, then there exist $($randomized$)$ stopping times that yield an infinite expected $($discounted $)$ reward; if $W<\infty$, then $V(x)=\lim_{y\to\infty}E_x\left(e^{-q\tau_y}g(X_{\tau_y})\mathbf{1}_{\{\tau_y<\infty\}}\right)$ and there is no optimal stopping time.
\end{itemize}
\end{thm}

To prove Theorem \ref{t3.1}, we need the following lemmas where $g\ge0$ is assumed to be nonconstant, increasing, logconcave and right-continuous. Some arguments are needed to deal with the case that $g$ is not continuous at $x_0=\inf\{s:g(s)>0\}$, which is not covered in Chapter 3 of Shiryaev \cite{ref17} where $g$ is required to satisfy $P_x(\mathop{\underline{\lim}}\limits_{t\downarrow0}g(X_t)\ge g(x))=1$ for all $x\in\mathbb{R}$.
\begin{lem}\label{l3.1}
For $x\in\mathbb{R}$ and $\ell=1,2,\dots$, we have $V^{(\ell+1)}(x)\ge V^{(\ell)}(x)>0$ and $u^{(\ell+1)}\ge u^{(\ell)}\ge x_0$ where $x_0:=\inf\{s:g(s)>0\}$.
\end{lem}
\begin{proof}
Since $\mathcal{M}^{(\ell)}\subset\mathcal{M}^{(\ell+1)}$, we have
\begin{align*}
0<V^{(\ell)}(x)&=\sup_{\tau\in\mathcal{M}^{(\ell)}}E_x\left(e^{-q\tau}g(X_\tau)\mathbf{1}_{\{\tau<\infty\}}\right)\\
&\le\sup_{\tau\in\mathcal{M}^{(\ell+1)}}E_x\left(e^{-q\tau}g(X_\tau)\mathbf{1}_{\{\tau<\infty\}}\right)=V^{(\ell+1)}(x).
\end{align*}
To show $u^{(\ell+1)}\ge u^{(\ell)}$, it suffices to consider the case $u^{(\ell+1)}<\infty$ (possibly $u^{(\ell+1)}=-\infty$). We have by Corollary \ref{c1}(i)
$$
g(x)=V^{(\ell+1)}(x)\ge V^{(\ell)}(x)\ge g(x)\;\;\mbox{for}\;\;x\ge u^{(\ell+1)},
$$
implying that $V^{(\ell)}(x)=g(x)$ for $x\ge u^{(\ell+1)}$. By Corollary \ref{c1}(i) again, we have $u^{(\ell+1)}\ge u^{(\ell)}$, completing the proof.
\end{proof}

\begin{lem}\label{l3.2}
Suppose $g(x)>0$ for all $x\in\mathbb{R}$. Then
\begin{itemize}
\item[\upshape(i)] $V^{(\infty)}(x):=\lim_{\ell\to\infty}V^{(\ell)}(x)=V(x)$ for $x\in\mathbb{R}$.
\item[\upshape(ii)] $V(x)>g(x)$ for $x<u$ and $V(x)=g(x)$ for $x\ge u$.
\item[\upshape(iii)] $V(x)/g(x)$ is decreasing in $x$.
\end{itemize}
\end{lem}
\begin{proof}
Since $g$ is assumed to be logconcave, $g(x)>0$ for all $x\in\mathbb{R}$ implies that $g$ is a continuous function.

(i) Since $V(x)\ge V^{(\ell)}(x)$, we have $V(x)\ge V^{(\infty)}(x)$. To show $V(x)\le V^{(\infty)}(x)$, let $\tau\in\mathcal{M}$ be any stopping time. Define for $\ell=1,2,\dots$,
$$
\tau^{(\ell)}=
\begin{cases}
2^{-\ell}(\lfloor2^{\ell}\tau\rfloor+1),&\;\;\mbox{if}\;\;\tau<\infty;\\
\infty,&\;\;\mbox{otherwise},
\end{cases}
$$
where $\lfloor x\rfloor$ denotes the largest integer not exceeding $x$. Clearly, $\tau^{(\ell)}\in\mathcal{M}^{(\ell)}$ and $\tau^{(\ell)}\searrow\tau$ as $\ell\to\infty$. It follows from the right-continuity of $\{X_t\}$ and continuity of $g$ that
$$
e^{-q\tau^{(\ell)}}g(X_{\tau^{(\ell)}})\to e^{-q\tau}g(X_\tau)\;\;\mbox{a.s. on}\;\;\{\tau<\infty\}.
$$
We have by Fatou's lemma that
\begin{align*}
V^{(\infty)}(x)&=\lim_{\ell\to\infty}V^{(\ell)}(x)\\
&=\sup_{\ell}V^{(\ell)}(x)\\
&\ge \sup_{\ell}E_x\left(e^{-q\tau^{(\ell)}}g(X_{\tau^{(\ell)}})\mathbf{1}_{\{\tau^{(\ell)}<\infty\}}\right)\\
&\ge\mathop{\underline{\lim}}_{\ell\to\infty}E_x\left(e^{-q\tau^{(\ell)}}g(X_{\tau^{(\ell)}})\mathbf{1}_{\{\tau^{(\ell)}<\infty\}}\right)\\
&\ge E_x\left(e^{-q\tau}g(X_\tau)\mathbf{1}_{\{\tau<\infty\}}\right).
\end{align*}
Since $\tau\in\mathcal{M}$ is arbitrary, we have $V(x)\le V^{(\infty)}(x)$.

(ii) For $x<u$, choose a (large) $\ell$ with $x<u^{(\ell)}\le u$, so that by Corollary \ref{c1}(i)
$$
g(x)<V^{(\ell)}(x)\le V^{(\infty)}(x)=V(x).
$$
For $x\ge u$, since $u^{(\ell)}\le u\le x$, we have by Corollary \ref{c1}(i)
$$
g(x)=V^{(\ell)}(x)\;\;\mbox{for all}\;\;\ell,
$$
implying that $g(x)=V^{(\infty)}(x)=V(x)$.

(iii) By Corollary \ref{c1}(ii), $V^{(\ell)}(x)/g(x)$ is decreasing in $x$. Since $V^{(\ell)}(x)\nearrow V(x)$ as $\ell\to\infty$, it follows that $V(x)/g(x)$ is decreasing in $x$.
\end{proof}

\begin{lem}\label{l3.3}
If $-\infty<u<\infty$, then $V(x)=g(x)$ for $x\ge u$.
\end{lem}
\begin{proof}
Let $h(x)=\log g(x)$ for $x\in\mathbb{R}$. Fix an $x>u\;(\ge x_0:=\inf\{s:g(s)>0\})$. If $h'(x-)=0$, then $h'(y-)=0$ for $y>x$, so that $g(x)=\sup\{g(y):y\in\mathbb{R}\}\ge V(x)$, implying that $V(x)=g(x)$. Suppose $h'(x-)>0$. Let
$$
\tilde{h}(y):=
\begin{cases}
h(x)+h'(x-)(y-x),&\;\;\mbox{if}\;\;y<x;\\
h(y),&\;\;\mbox{otherwise}.
\end{cases}
$$
Let $\tilde{g}(y)=e^{\tilde{h}(y)}>0$ for all $y\in\mathbb{R}$, which is larger than or equal to $g(y)$, nonconstant, increasing, logconcave and continuous. (Note that $\tilde{g}(y)$ is nonconstant since $h'(x-)>0$.) Define
$$
\tilde{u}^{(\ell)}:=\inf\left\{y\in\mathbb{R}:\frac{E_y\left(e^{-qT^{(\ell)}_{y}}\tilde{g}(X_{T^{(\ell)}_y})\mathbf{1}_{\{T^{(\ell)}_y<\infty\}}\right)}{\tilde{g}(y)}\le1\right\}.
$$
Since $\tilde{g}(y)=g(y)$ for all $y\ge x$ and $u^{(\ell)}\le u<x$, we have $\tilde{u}^{(\ell)}\le x$. Let $\tilde{u}:=\lim_{\ell\to\infty}\tilde{u}^{(\ell)}\le x$. By Lemma \ref{l3.2}(ii) applied to $\tilde{g}$,
$$
g(x)=\tilde{g}(x)=\tilde{V}(x):=\sup_{\tau\in\mathcal{M}}E_x\left(e^{-q\tau}\tilde{g}(X_\tau)\mathbf{1}_{\{\tau<\infty\}}\right)\ge V(x),
$$
where the inequality follows from $\tilde{g}(y)\ge g(y)$ for all $y\in\mathbb{R}$. So $V(x)=g(x)$. We have shown $V(x)=g(x)$ for all $x>u$. Furthermore,
$$
g(u)\le V(u)\le\lim_{x\downarrow u}V(x)=\lim_{x\downarrow u}g(x)=g(u),
$$
implying that $V(u)=g(u)$, completing the proof.
\end{proof}

\begin{lem}\label{l3.4}
Suppose $-\infty<u<\infty$. For $x<u$ and $\tau\in\mathcal{M}$ with $P_x(\tau\ge\tau_u)=1$, we have
$$
E_x\left(e^{-q\tau}g(X_\tau)\mathbf{1}_{\{\tau<\infty\}}\right)\le E_x\left(e^{-q\tau_{u}}g(X_{\tau_{u}})\mathbf{1}_{\{\tau_{u}<\infty\}}\right).
$$
\end{lem}
\begin{proof}
Since on $\{\tau_u<\infty\}$, $X_{\tau_u+s}-X_{\tau_u}$ ($s\ge0$) is independent of $\mathcal{F}_{\tau_u}$ and has the same law as $X_s-X_0$, we have
\begin{align*}
E_x\left[e^{-q(\tau-\tau_u)}g(X_{\tau_u}+(X_\tau-X_{\tau_u}))\mathbf{1}_{\{\tau-\tau_u<\infty\}}\mid\mathcal{F}_{\tau_u}\right]&\le V(X_{\tau_u})=g(X_{\tau_u})\;\;\mbox{a.s.},
\end{align*}
where the equality follows from Lemma \ref{l3.3} (noting that $X_{\tau_u}\ge u$ on $\{\tau_u<\infty\}$). So,
\begin{align*}
E_x\left(e^{-q\tau}g(X_\tau)\mathbf{1}_{\{\tau<\infty\}}\right)&=E_x\left\{e^{-q\tau_u}\mathbf{1}_{\{\tau_u<\infty\}}E_x\left[e^{-q(\tau-\tau_u)}g(X_\tau)\mathbf{1}_{\{\tau-\tau_u<\infty\}}\mid\mathcal{F}_{\tau_u}\right]\right\}\\
&\le E_x\left(e^{-q\tau_u}g(X_{\tau_u})\mathbf{1}_{\{\tau_u<\infty\}}\right),
\end{align*}
completing the proof.
\end{proof}

\begin{lem}\label{l3.5}
Suppose $-\infty<u<\infty$. For $x<u$,
$$
V^{(\ell)}(x-u+u^{(\ell)})\le E_x\left(e^{-q\tau_u}g(X_{\tau_u})\mathbf{1}_{\{\tau_u<\infty\}}\right).
$$
\end{lem}
\begin{proof}
For ease of notation, write $v=u^{(\ell)}$ and $\delta=u-u^{(\ell)}=u-v\ge0$. Noting that for $x<u$,
$$
\mathcal{L}(\tau^{(\ell)}_u,X_{\tau^{(\ell)}_u}\mathbf{1}_{\{\tau^{(\ell)}_u<\infty\}}\mid X_0=x)=\mathcal{L}(\tau^{(\ell)}_v,(\delta+X_{\tau^{(\ell)}_v})\mathbf{1}_{\{\tau^{(\ell)}_v<\infty\}}\mid X_0=x-\delta),
$$
we have
\begin{align*}
V^{(\ell)}(x-\delta)&=E_{x-\delta}\left(e^{-q\tau^{(\ell)}_v}g(X_{\tau^{(\ell)}_v})\mathbf{1}_{\{\tau^{(\ell)}_v<\infty\}}\right)\\
&=E_{x}\left(e^{-q\tau^{(\ell)}_u}g(X_{\tau^{(\ell)}_u}-\delta)\mathbf{1}_{\{\tau^{(\ell)}_u<\infty\}}\right)\\
&\le E_{x}\left(e^{-q\tau^{(\ell)}_u}g(X_{\tau^{(\ell)}_u})\mathbf{1}_{\{\tau^{(\ell)}_u<\infty\}}\right)\\
&\le
E_{x}\left(e^{-q\tau_u}g(X_{\tau_u})\mathbf{1}_{\{\tau_u<\infty\}}\right),
\end{align*}
where the last inequality follows from Lemma \ref{l3.4} and the fact that $P_x(\tau^{(\ell)}_u\ge\tau_u)=1$.
\end{proof}

\begin{lem}\label{l3.6}
Suppose $x_0<u<\infty$. Then $V(x)$ is continuous everywhere and $\tau_u$ is optimal.
\end{lem}
\begin{proof}
Fix $v\in(x_0,u)$. Let $h(x)=\log g(x)$, and
$$
\tilde{h}(x):=
\begin{cases}
h(v)+h'(v-)(x-v),&\;\;\mbox{if}\;\;x<v;\\
h(x),&\;\;\mbox{otherwise},
\end{cases}
$$
and $\tilde{g}(x):=e^{\tilde{h}(x)}>0$, so that $\tilde{g}(x)\ge g(x)$ for $x\in\mathbb{R}$. Clearly, $\tilde{g}(x)$ is nonconstant, increasing, logconcave and continuous. Define
$\tilde{u}^{(\ell)}$, $\tilde{u}$, $\tilde{V}^{(\ell)}(x)$ and $\tilde{V}(x)$ in terms of $\tilde{g}$ in exactly the same way that $u^{(\ell)}$, $u$, $V^{(\ell)}(x)$ and $V(x)$ are defined in terms of $g$. For (large) $\ell$ with $v<u^{(\ell)}$, the fact that $\tilde{g}(x)=g(x)$ for all $x\ge v$ yields $\tilde{u}^{(\ell)}=u^{(\ell)}$, implying that
$$
\tilde{u}:=\lim_{\ell\to\infty}\tilde{u}^{(\ell)}=\lim_{\ell\to\infty}u^{(\ell)}=u.
$$
Moreover, $\tilde{u}^{(\ell)}=u^{(\ell)}>v$ and $\tilde{g}(x)=g(x)$ for all $x\ge v$ implies that $\tilde{V}^{(\ell)}(x)=V^{(\ell)}(x)$ for $x\in\mathbb{R}$, which in turn implies that for $x\in\mathbb{R}$,
\begin{align*}
V(x)&\le\tilde{V}(x)\\
&=\lim_{\ell\to\infty}\tilde{V}^{(\ell)}(x)\;\;\mbox{(by Lemma \ref{l3.2}(i) applied to $\tilde{g}$)}\\
&=\lim_{\ell\to\infty}V^{(\ell)}(x)\le V(x).
\end{align*}
It follows that $V(x)=\tilde{V}(x)$ for $x\in\mathbb{R}$.

By Lemma \ref{l3.2}(iii) (applied to $\tilde{g}$), $\tilde{V}(x)/\tilde{g}(x)$ is decreasing in $x$, implying that
$$
\frac{\tilde{V}(x-)}{\tilde{g}(x-)}\ge\frac{\tilde{V}(x+)}{\tilde{g}(x+)},\;\;x\in\mathbb{R}.
$$
Since $\tilde{g}(x)=\tilde{g}(x-)=\tilde{g}(x+)>0$, we have $\tilde{V}(x-)\ge \tilde{V}(x+)$, implying that $\tilde{V}(x-)=\tilde{V}(x+)$. So $\tilde{V}(x)$ $(=V(x))$ is a continuous function.

To show the optimality of $\tau_u$, we need to prove
\begin{equation}\label{d1}
V(x)=E_x\left(e^{-q\tau_u}g(X_{\tau_u})\mathbf{1}_{\{\tau_u<\infty\}}\right)\;\;\mbox{for}\;\;x\in\mathbb{R}.
\end{equation}
By Lemma \ref{l3.3}, $V(x)=g(x)$ for $x\ge u$, so that (\ref{d1}) holds for $x\ge u$. For $x<u (=\tilde{u})$, we have by Lemma \ref{l3.5} (applied to $\tilde{g}$)
$$
E_x\left(e^{-q\tau_u}g(X_{\tau_u})\mathbf{1}_{\{\tau_u<\infty\}}\right)=E_x\left(e^{-q\tau_u}\tilde{g}(X_{\tau_u})\mathbf{1}_{\{\tau_u<\infty\}}\right)\ge \tilde{V}^{(\ell)}(x-u+\tilde{u}^{(\ell)}).
$$
It follows that
\begin{align*}
E_x\left(e^{-q\tau_u}g(X_{\tau_u})\mathbf{1}_{\{\tau_u<\infty\}}\right)&\ge\sup_{\ell}\tilde{V}^{(\ell)}(x-u+\tilde{u}^{(\ell)})\\
&\ge\lim_{\varepsilon\downarrow0}\lim_{\ell\to\infty}\tilde{V}^{(\ell)}(x-\varepsilon)\\
&=\tilde{V}(x-)\;\;\mbox{(by Lemma \ref{l3.2}(i) applied to $\tilde{g}$)}\\
&=\tilde{V}(x)=V(x)\ge E_x\left(e^{-q\tau_u}g(X_{\tau_u})\mathbf{1}_{\{\tau_u<\infty\}}\right).
\end{align*}
This establishes (\ref{d1}) and completes the proof.
\end{proof}

\begin{lem}\label{l3.7}
Suppose that $-\infty<u<\infty$ and $\tau_u$ is optimal. Then
\begin{itemize}
\item[\upshape(i)] $g(x)\le E_{x}\left(e^{-q\tau_y}g(X_{\tau_y})\mathbf{1}_{\{\tau_y<\infty\}}\right)$ for $x<y\le u$;
\item[\upshape(ii)]
$E_{x}\left(e^{-q\tau_y}g(X_{\tau_y})\mathbf{1}_{\{\tau_y<\infty\}}\right)\le E_{x}\left(e^{-q\tau_z}g(X_{\tau_z})\mathbf{1}_{\{\tau_z<\infty\}}\right)$ for $x\le y<z\le u$.
\end{itemize}
\end{lem}
\begin{proof}
(i) The desired inequality holds trivially if $g(x)=0$. Suppose $g(x)>0$. Since $\mathcal{L}(\tau_y,(u-y+X_{\tau_y})\mathbf{1}_{\{\tau_y<\infty\}}\mid X_0=x)=\mathcal{L}(\tau_u,X_{\tau_u}\mathbf{1}_{\{\tau_u<\infty\}}\mid X_0=x+u-y)$,
we have
\begin{align*}
1&\le\frac{V(x+u-y)}{g(x+u-y)}\\
&=\frac{E_{x+u-y}\left(e^{-q\tau_u}g(X_{\tau_u})\mathbf{1}_{\{\tau_u<\infty\}}\right)}{g(x+u-y)}\\
&=\frac{E_{x}\left(e^{-q\tau_y}g(u-y+X_{\tau_y})\mathbf{1}_{\{\tau_y<\infty\}}\right)}{g(u-y+x)}\\
&\le\frac{E_{x}\left(e^{-q\tau_y}g(X_{\tau_y})\mathbf{1}_{\{\tau_y<\infty\}}\right)}{g(x)},
\end{align*}
where the last inequality follows from the logconcavity of $g$.

(ii) Noting that $\tau_z\ge\tau_y$ a.s., we have
\begin{equation}\label{b3}
E_{x}\left(e^{-q\tau_z}g(X_{\tau_z})\mathbf{1}_{\{\tau_z<\infty\}}\right)=E_{x}\left\{e^{-q\tau_y}\mathbf{1}_{\{\tau_y<\infty\}}E_x\left[e^{-q(\tau_z-\tau_y)}g(X_{\tau_y}+(X_{\tau_z}-X_{\tau_y}))\mathbf{1}_{\{\tau_z-\tau_y<\infty\}}\mid\mathcal{F}_{\tau_y}\right]\right\}.
\end{equation}
Since on $\{\tau_y<\infty\}$, $X_{\tau_y+s}-X_{\tau_y}$ is independent of $\mathcal{F}_{\tau_y}$ and has the same law as $X_s-X_0$, we have by part (i) that on $\{\tau_y<\tau_z\}=\{\tau_y<\infty,X_{\tau_y}<z\}$
$$
E_x\left[e^{-q(\tau_z-\tau_y)}g(X_{\tau_y}+(X_{\tau_z}-X_{\tau_y}))\mathbf{1}_{\{\tau_z-\tau_y<\infty\}}\mid\mathcal{F}_{\tau_y}\right]\ge g(X_{\tau_y})\;\;\mbox{a.s.,}
$$
which together with (\ref{b3}) yields the desired inequality.
\end{proof}

\begin{proof}[\bf Proof of Theorem \ref{t3.1}(i)]
If $u=-\infty$, then $u^{(\ell)}=-\infty$ for $\ell=1,2,\dots$, implying that $0<V^{(\ell)}(x)=g(x)$ for all $x\in\mathbb{R}$ and $\ell=1,2,\dots$. It follows from Lemma \ref{l3.2}(i) that
$$
V(x)=\lim_{\ell\to\infty}V^{(\ell)}(x)=g(x)\;\;\mbox{for}\;\;x\in\mathbb{R},
$$
proving that $\tau_u=\tau_{-\infty}$ is optimal.

For $-\infty<u<\infty$, the optimality of $\tau_u$ is established in Lemma \ref{l3.6} if $u>x_0$. It remains to show that for $u=x_0(>-\infty)$,
\begin{equation}\label{b4}
V(x)=E_x\left(e^{-q\tau_u}g(X_{\tau_u})\mathbf{1}_{\{\tau_u<\infty\}}\right)\;\;\mbox{for}\;\;x\in\mathbb{R}.
\end{equation}
By Lemma \ref{l3.3}, $V(x)=g(x)$ for $x\ge u=x_0$, so that (\ref{b4}) holds for $x\ge u=x_0$. For $x<u=x_0$ and any stopping time $\tau\in\mathcal{M}$, we have $X_\tau<x_0$ a.s. on $\{\tau<\tau_{x_0}\}$, so that
\begin{equation}\label{d2}
g(X_\tau)\mathbf{1}_{\{\tau<\tau_{x_0}\}}=0\;\;\mbox{a.s. and}\;\;E_x\left(e^{-q\tau}g(X_\tau)\mathbf{1}_{\{\tau<\tau_{x_0}\}}\right)=0.
\end{equation}
Furthermore, we have $X_{\tau_{x_0}}\ge x_0=u$ a.s. on $\{\tau_{x_0}<\tau\}$, so that by Lemma \ref{l3.3} (with $u=x_0$), on $\{\tau_{x_0}<\tau\}$,
\begin{align*}
g(X_{\tau_{x_0}})&=V(X_{\tau_{x_0}})\\
&\ge E_x\left[e^{-q(\tau-\tau_{x_0})}g\left(X_{\tau_{x_0}}+(X_\tau-X_{\tau_{x_0}})\right)\mathbf{1}_{\{\tau_{x_0}<\tau<\infty\}}\mid\mathcal{F}_{\tau_{x_0}}\right].
\end{align*}
It follows that
$$
E_x\left(e^{-q\tau_{x_0}}g(X_{\tau_{x_0}})\mathbf{1}_{\{\tau_{x_0}<\tau\}}\right)\ge E_x\left(e^{-q\tau}g(X_\tau)\mathbf{1}_{\{\tau_{x_0}<\tau<\infty\}}\right),
$$
which together with (\ref{d2}) implies that
$$
E_x\left(e^{-q\tau_{x_0}}g(X_{\tau_{x_0}})\mathbf{1}_{\{\tau_{x_0}<\infty\}}\right)\ge E_x\left(e^{-q\tau}g(X_\tau)\mathbf{1}_{\{\tau<\infty\}}\right).
$$
Since $\tau\in\mathcal{M}$ is arbitrary, (\ref{b4}) follows. The proof is complete.
\end{proof}

\begin{rem}
If $u=x_0>-\infty$, then $V(x)$ is not necessarily continuous.
\end{rem}


\begin{proof}[\bf Proof of Theorem \ref{t3.1}(ii)]
Assume $u=\infty$. We claim that
\begin{equation}\label{b5}
E_x\left(e^{-q\tau_y}g(X_{\tau_y})\mathbf{1}_{\{\tau_y<\infty\}}\right)\;\;\mbox{is increasing in}\;\;y.
\end{equation}
Consider an increasing sequence $\{b_k\}$ satisfying $b_1>x_0$ and $\lim_{k\to\infty}b_k=\infty$. Let $g_k(x)=g(x\wedge b_k)$, which is nonconstant, increasing, logconcave and right-continuous. For $\ell=1,2,\dots$, let
$$
u_k^{(\ell)}:=\inf\left\{x\in\mathbb{R}:\frac{E_x\left(e^{-qT_x^{(\ell)}}g_k(X_{T_x^{(\ell)}})\mathbf{1}_{\{T_x^{(\ell)}<\infty\}}\right)}{g_k(x)}\le1\right\}\le b_k.
$$
Let $u_k=\lim_{\ell\to\infty}u_{k}^{(\ell)}\le b_k<\infty$. In other words, $u_k$ is defined in terms of $g_k$ in exactly the same way that $u$ is defined in terms of $g$. Since $u_k<\infty$, we have by Theorem \ref{t3.1}(i)
\begin{align}
V_k(x)&:=\sup_{\tau\in\mathcal{M}}E_x\left(e^{-q\tau}g_k(X_{\tau})\mathbf{1}_{\{\tau<\infty\}}\right)\notag\\
&=E_x\left(e^{-q\tau_{u_k}}g_k(X_{\tau_{u_k}})\mathbf{1}_{\{\tau_{u_k}<\infty\}}\right),\;\;x\in\mathbb{R}\label{b6}.
\end{align}
Thus, $V_k(x)=g_k(x)$ for $x\ge u_k$ and $V_k(x)>g_k(x)$ for $x<u_k$. (Note that by Corollary \ref{c1}(i),
$g_k(x)<V^{(\ell)}_k(x)\le V_k(x)$ for $x<u^{(\ell)}_k\le u_k$ where $V^{(\ell)}_k(x)$ is the supremum of $E_x\left(e^{-q\tau}g_k(X_\tau)\mathbf{1}_{\{\tau<\infty\}}\right)$ over $\tau\in\mathcal{M}^{(\ell)}$.) Since $g_k$ is increasing in $k$, it is easily shown that both $V_k$ and $u_k$ are increasing. Let $V_\infty(x)=\lim_{k\to\infty}V_k(x)$ and $u_\infty=\lim_{k\to\infty}u_k$. Then clearly $V_\infty(x)=V(x)$. Since $u^{(\ell)}_k\nearrow u^{(\ell)}$ as $k\to\infty$ and $u^{(\ell)}_k\nearrow u_k$ as $\ell\to\infty$, it follows that $u=\lim_{\ell\to\infty}u^{(\ell)}=\infty$ implies $u_\infty=\lim_{k\to\infty}u_k=\infty$. Incidentally, since $V(x)\ge V_k(x)>g_k(x)=g(x)$ for $x<u_k\le b_k$, we have \begin{equation}\label{d5}
V(x)>g(x)\;\;\mbox{for all}\;\;x\in\mathbb{R}.
\end{equation}

To prove (\ref{b5}), we need to show for $x\le y_1<y_2$ that
\begin{equation}\label{b7}
E_x\left(e^{-q\tau_{y_1}}g(X_{\tau_{y_1}})\mathbf{1}_{\{\tau_{y_1}<\infty\}}\right)\le E_x\left(e^{-q\tau_{y_2}}g(X_{\tau_{y_2}})\mathbf{1}_{\{\tau_{y_2}<\infty\}}\right).
\end{equation}
For large $k$ with $u_k>y_2$, applying Lemma \ref{l3.7} to $g_k$ yields
\begin{equation}\label{b8}
E_x\left(e^{-q\tau_{y_1}}g_k(X_{\tau_{y_1}})\mathbf{1}_{\{\tau_{y_1}<\infty\}}\right)\le E_x\left(e^{-q\tau_{y_2}}g_k(X_{\tau_{y_2}})\mathbf{1}_{\{\tau_{y_2}<\infty\}}\right).
\end{equation}
By the monotone convergence theorem, the two sides of (\ref{b8}) converge to the corresponding sides of (\ref{b7}), respectively. This proves (\ref{b7}) and establishes the claim (\ref{b5}).

By (\ref{b5}),
\begin{align*}
W&:=\sup_{y\in\mathbb{R}}E\left(e^{-q\tau_{y}}g(X_{\tau_{y}})\mathbf{1}_{\{\tau_{y}<\infty\}}\right)\\ &=\lim_{y\to\infty}E\left(e^{-q\tau_{y}}g(X_{\tau_{y}})\mathbf{1}_{\{\tau_{y}<\infty\}}\right).
\end{align*}
Following the argument for (\ref{v7}) in the proof of Theorem \ref{t2.1}(ii), we can show that
\begin{align}
Q(x)&:=\sup_{y\in\mathbb{R}}E_{x}\left(e^{-q\tau_{y}}g(X_{\tau_{y}})\mathbf{1}_{\{\tau_{y}<\infty\}}\right)\notag\\ &=\lim_{y\to\infty}E_{x}\left(e^{-q\tau_{y}}g(X_{\tau_{y}})\mathbf{1}_{\{\tau_{y}<\infty\}}\right)=e^{\beta x}W.\label{d3}
\end{align}
Furthermore, we have
\begin{align*}
V(x)=V_\infty(x)&=\lim_{k\to\infty}V_k(x)\\
&=\lim_{k\to\infty}E_x\left(e^{-q\tau_{u_k}}g_k(X_{\tau_{u_k}})\mathbf{1}_{\{\tau_{u_k}<\infty\}}\right)\;\;\;\mbox{(by (\ref{b6}))}\\
&\le\lim_{k\to\infty}E_x\left(e^{-q\tau_{u_k}}g(X_{\tau_{u_k}})\mathbf{1}_{\{\tau_{u_k}<\infty\}}\right)\\
&=\lim_{y\to\infty}E_{x}\left(e^{-q\tau_{y}}g(X_{\tau_{y}})\mathbf{1}_{\{\tau_{y}<\infty\}}\right)\;\;\;\mbox{(by (\ref{b5}) and}\;u_k\nearrow\infty)\\
&=e^{\beta x}W\;\;\;\mbox{(by (\ref{d3}))}\\
&\le V(x),
\end{align*}
implying that
\begin{equation}\label{d4}
V(x)=Q(x)=e^{\beta x}W\;\;\mbox{for}\;\;x\in\mathbb{R}.
\end{equation}
If $W=\infty$, it is readily seen that there are randomized stopping times that yield an infinite expected (discounted) reward. Suppose $W<\infty$. To show
\begin{equation}\label{d6}
V(x)>E\left(e^{-q\tau}g(X_{\tau})\mathbf{1}_{\{\tau<\infty\}}\right)\;\;\mbox{for all}\;\;\tau\in\mathcal{M},
\end{equation}
we first prove that $\{e^{-qt}Q(X_t)=e^{\beta X_t-qt}W\}_{t\ge0}$ is a (positive) supermartingale, or equivalently, $E(e^{\beta X_t-qt})\le1$ for $t>0$. Since $E(e^{\beta X_t-qt})=(E(e^{\beta X_1-q}))^{t}$, it suffices to show $E(e^{\beta X_1-q})\le1$. Suppose to the contrary that $E(e^{\beta X_1-q})>1$. Then for any $x$ with $h(x)=\log g(x)>-\infty$, let
$$
\tilde{h}(y)=h(x)+\beta(y-x)\;\;\mbox{for}\;\;y\in\mathbb{R}.
$$
Since $h'(y-)\ge\beta$ for $y\in\mathbb{R}$, we have $\tilde{h}(y)\le h(y)$ and $\tilde{g}(y):=e^{\tilde{h}(y)}\le g(y)$ for $y\ge x$. Then
\begin{align*}
\infty>e^{\beta x}W=V(x)&\ge E_x(e^{-qt}g(X_t))\\
&\ge E_x\left(e^{-qt}g(X_t)\mathbf{1}_{\{X_t\ge x\}}\right)\\
&\ge E_x\left(e^{-qt}\tilde{g}(X_t)\mathbf{1}_{\{X_t\ge x\}}\right)\\
&=E_x\left(e^{-qt}\tilde{g}(X_t)\right)-E_x\left(e^{-qt}\tilde{g}(X_t)\mathbf{1}_{\{X_t<x\}}\right)\\
&=E_x\left(e^{-qt+h(x)+\beta (X_t-x)}\right)-E_x\left(e^{-qt}\tilde{g}(X_t)\mathbf{1}_{\{X_t<x\}}\right)\\
&\ge g(x)\left(E(e^{\beta X_1-q})\right)^{t}-g(x)\to\infty\;\;\mbox{as}\;\;t\to\infty,
\end{align*}
a contradiction. So $\{e^{\beta X_t-qt}W\}_{t\ge0}$ is a positive supermartingale. To show (\ref{d6}), it suffices to consider $\tau\in\mathcal{M}$ with $P(\tau<\infty)>0$. Then
\begin{align*}
V(x)&=e^{\beta x}W\\
&\ge E_x\left(e^{\beta X_{\tau}-q\tau}W\mathbf{1}_{\{\tau<\infty\}}\right)\\
&=E_x\left(e^{-q\tau}V(X_{\tau})\mathbf{1}_{\{\tau<\infty\}}\right)\\
&>E_x\left(e^{-q\tau}g(X_{\tau})\mathbf{1}_{\{\tau<\infty\}}\right),
\end{align*}
where the last (strict) inequality follows from (\ref{d5}) and $P(\tau<\infty)>0$. The proof is complete.
\end{proof}

\begin{rem}
Previously Lemma \ref{l3.2} was established under the assumption of $x_0=-\infty$. Examining the proof of Theorem \ref{t3.1} shows that Lemma \ref{l3.2} remains true for $x_0>-\infty$. In particular, we have $V(x)>g(x)$ for $x<u$ and $V(x)=g(x)$ for $x\ge u$.
\end{rem}

\section{On the principle of smooth fit for L\'{e}vy processes}
\hspace*{18pt}In this section we investigate the principle of smooth fit for L\'{e}vy processes. Let $X=\{X_t\}_{t\ge0}$ be a L\'{e}vy process with initial state $X_0=x\in\mathbb{R}$ and assume $P(X_1>0)>0$. For $y\in\mathbb{R}$, let
$\tau_y=\inf\{t\ge0:X_t\ge y\}$ and $\tau_{y+}=\inf\{t\ge0:X_t>y\}$.
\begin{thm}\label{g6}
Let $g:\mathbb{R}\to[0,\infty)$ be nonconstant, increasing, logconcave and right-continuous. Define $V(x)$ and $u$ as in $(\ref{e3.1})$ and $(\ref{b2})$. Suppose $-\infty\le x_0<u<\infty$ and $g$ is differentiable at $u$ $(\mbox{i.e.}\; g'(u-)=g'(u+)=g'(u))$. If $0$ is regular for $(0,\infty)$ for $X$, then $V$ is differentiable at $u$, i.e. $V'(u-)=V'(u+)$ $(=g'(u))$.
\end{thm}
\begin{proof}
Since $V(u-\varepsilon)>g(u-\varepsilon)$ for $\varepsilon>0$ and $V(u)=g(u)$, we have
\begin{equation}\label{d11}
\overline{\lim_{\varepsilon\downarrow0}}\left(\frac{V(u)-V(u-\varepsilon)}{\varepsilon}\right)\le\overline{\lim_{\varepsilon\downarrow0}}\left(\frac{g(u)-g(u-\varepsilon)}{\varepsilon}\right)=g'(u-)=g'(u).
\end{equation}
Since $\mathcal{L}(\tau_{u+\varepsilon},X_{\tau_{u+\varepsilon}}\mathbf{1}_{\{\tau_{u+\varepsilon}<\infty\}}\mid X_0=u)=\mathcal{L}(\tau_{\varepsilon},(u+X_{\tau_{\varepsilon}})\mathbf{1}_{\{\tau_{\varepsilon}<\infty\}}\mid X_0=0)$ and
$$
\mathcal{L}(\tau_{u},X_{\tau_{u}}\mathbf{1}_{\{\tau_{u}<\infty\}}\mid X_0=u-\varepsilon)=\mathcal{L}(\tau_{\varepsilon},(u-\varepsilon+X_{\tau_{\varepsilon}})\mathbf{1}_{\{\tau_{\varepsilon}<\infty\}}\mid X_0=0),
$$
we have
\begin{align}
\frac{V(u)-V(u-\varepsilon)}{\varepsilon}&\ge\frac{1}{\varepsilon}\left[E_u\left(e^{-q\tau_{u+\varepsilon}}g(X_{\tau_{u+\varepsilon}})\mathbf{1}_{\{\tau_{u+\varepsilon}<\infty\}}\right)-E_{u-\varepsilon}\left(e^{-q\tau_{u}}g(X_{\tau_{u}})\mathbf{1}_{\{\tau_{u}<\infty\}}\right)\right]\notag\\
&=\frac{1}{\varepsilon}\left[E\left(e^{-q\tau_{\varepsilon}}g(u+X_{\tau_{\varepsilon}})\mathbf{1}_{\{\tau_{\varepsilon}<\infty\}}\right)-E\left(e^{-q\tau_{\varepsilon}}g(u-\varepsilon+X_{\tau_{\varepsilon}})\mathbf{1}_{\{\tau_{\varepsilon}<\infty\}}\right)\right]\notag\\
&=\frac{1}{\varepsilon}E\left[e^{-q\tau_{\varepsilon}}\left(g(u+X_{\tau_{\varepsilon}})-g(u-\varepsilon+X_{\tau_{\varepsilon}})\right)\mathbf{1}_{\{\tau_{\varepsilon}<\infty\}}\right]\label{d12}.
\end{align}
By the concavity of $h(x)=\log g(x)$, we have on $\{\tau_{\varepsilon}<\infty\}$
\begin{align}
g(u+X_{\tau_{\varepsilon}})-g(u-\varepsilon+X_{\tau_{\varepsilon}})&=e^{h(u-\varepsilon+X_{\tau_{\varepsilon}})}\left[e^{h(u+X_{\tau_{\varepsilon}})-h(u-\varepsilon+X_{\tau_{\varepsilon}})}-1\right]\notag\\
&\ge e^{h(u-\varepsilon+X_{\tau_{\varepsilon}})}\left[e^{\varepsilon h'((u+X_{\tau_{\varepsilon}})+)}-1\right]\notag\\
&=e^{h(u-\varepsilon+X_{\tau_{\varepsilon}})}e^{\theta\varepsilon h'((u+X_{\tau_{\varepsilon}})+)}\varepsilon h'((u+X_{\tau_{\varepsilon}})+)\label{d13}
\end{align}
for some $\theta\in(0,1)$ by the mean value theorem applied to the function $e^x$. It follows from
(\ref{d12}) and (\ref{d13}) that
\begin{align}
\mathop{\underline{\lim}}_{\varepsilon\downarrow0}\left(\frac{V(u)-V(u-\varepsilon)}{\varepsilon}\right)&\ge\mathop{\underline{\lim}}_{\varepsilon\downarrow0}E\left(e^{-q\tau_{\varepsilon}+h(u-\varepsilon+X_{\tau_{\varepsilon}})+\theta\varepsilon h'((u+X_{\tau_{\varepsilon}})+)}h'((u+X_{\tau_{\varepsilon}})+)\mathbf{1}_{\{\tau_{\varepsilon}<\infty\}}\right)\notag\\
&\ge E\left(\mathop{\underline{\lim}}_{\varepsilon\downarrow0}e^{-q\tau_{\varepsilon}+h(u-\varepsilon+X_{\tau_{\varepsilon}})+\theta\varepsilon h'((u+X_{\tau_{\varepsilon}})+)}h'((u+X_{\tau_{\varepsilon}})+)\mathbf{1}_{\{\tau_{\varepsilon}<\infty\}}\right)\notag\\
&=E\left(e^{-q\tau_{0+}}e^{h(u+X_{\tau_{0+}})}h'((u+X_{\tau_{0+}})+)\mathbf{1}_{\{\tau_{0+}<\infty\}}\right)\notag\\
&=e^{h(u)}h'(u+)\label{d14},
\end{align}
where the second-to-last equality follows from the fact that $\tau_{\varepsilon}\downarrow\tau_{0+}$ as $\varepsilon\downarrow0$ together with the right-continuity of $\{X_t\}$ and the concavity of $h$, and the last equality follows from $P(\tau_{0+}=0)=1$ (since $0$ is regular for $(0,\infty)$). Combining (\ref{d11}) and (\ref{d14}) together with $e^{h(u)}h'(u+)=g'(u+)=g'(u)$ yields that $V'(u-)=g'(u)=V'(u+)$. The proof is complete.
\end{proof}

\begin{rem}
For a L\'{e}vy process $X$ with $0$ regular for $(0,\infty)$, Theorem \ref{g6} shows that the smooth fit principle holds if $g$ is differentiable at $u$ $($the optimal stopping boundary$)$. It is easy to show by example that the value function may fail to satisfy the smooth fit condition if $g'(u-)\neq g'(u+)$. Let $X$ be standard Brownian motion. For $($fixed$)$ $q>0$, consider $g(x)=\mathbf{1}_{[0,\infty)}(x)$, for which we have $u=0$. For $x<0$,
\begin{align*}
V(x)&=E_x(e^{-q\tau_0})\\
&=E(e^{-q\tau_{-x}})=e^{x\sqrt{2q}}.
\end{align*}
Then $V(x)=\min\{e^{x\sqrt{2q}},1\}$. More generally, for any $g^*(x)$ with $g^*(x)=g(x)$ for $x\ge0$ and $g^*(x)<V(x)$ for $x<0$, it is readily shown that $\tau_0$ is the optimal stopping time and $V(x)=\min\{e^{x\sqrt{2q}},1\}$ is the value function. For example, the reward function $g^*(x)=\min\{e^{rx},1\}$ with $\sqrt{2q}<r<\infty$ is increasing and logconcave with $(g^*)'(0-)=r>0=(g^*)'(0+)$, for which the value function $V(x)=\min\{e^{x\sqrt{2q}},1\}$ is not differentiable at $0$.
\end{rem}

Next, we examine the principle of smooth fit for L\'{e}vy processes with $0$ irregular for $(0,\infty)$.
\begin{thm}\label{g3}
Let $g:\mathbb{R}\to[0,\infty)$ be nonconstant, increasing, logconcave and right-continuous. Let $h(x)=\log g(x)$. Define $V(x)$ and $u$ as in $(\ref{e3.1})$ and $(\ref{b2})$. Suppose $-\infty\le x_0<u<\infty$ and $0$ is irregular for $(0,\infty)$ for $X$. Then
\begin{enumerate}
\item[\upshape(i)] $V'(u-)=E\left(e^{-q\tau_{0+}}g'((u+X_{\tau_{0+}})-)\mathbf{1}_{\{\tau_{0+}<\infty\}}\right)$.
\item[\upshape(ii)] $V'(u-)=V'(u+)$ $(=g'(u+))$ if and only if
    \begin{equation}\label{a1}
    h'((u+\zeta)-)=h'(u+)
    \end{equation}
where $\zeta=\inf\{x:P(X_{\tau_{0+}}>x\mid\tau_{0+}<\infty)=0\}$, the essential supremum of the $($conditional$)$ distribution $\mathcal{L}(X_{\tau_{0+}}\mid X_0=0,\tau_{0+}<\infty)$, and where $h'((u+\zeta)-)=\lim_{x\to\infty}h'(x-)$ if $\zeta=\infty$.
\item[\upshape(iii)] $V(x)=g(u)e^{h'(u+)(x-u)}$ for $x<u$, provided that condition $(\ref{a1})$ holds.
\end{enumerate}
\end{thm}

To prove Theorem \ref{g3}, we need the following lemmas.
\begin{lem}\label{g4}
Let $g:\mathbb{R}\to[0,\infty)$ be nonconstant, increasing, logconcave and right-continuous. Suppose $0$ is irregular for $(0,\infty)$ for $X$ and
\begin{equation}\label{d24}
\frac{E_x\left(e^{-q\tau_{x+}}g(X_{\tau_{x+}})\mathbf{1}_{\left\{\tau_{x+}<\infty\right\}}\right)}{g(x)}\le1\;\;\mbox{for some}\;\;x>x_0.
\end{equation}
Then
$$
g(x)\ge E_x\left(e^{-q\tau_{y}}g(X_{\tau_{y}})\mathbf{1}_{\left\{\tau_{y}<\infty\right\}}\right)\;\;\mbox{for all}\;\;y>x.
$$
\end{lem}
\begin{proof}
The following proof is similar to those of Lemmas \ref{lc} and \ref{ld}. Note that
$$
\frac{E_z\left(e^{-q\tau_{z+}}g(X_{\tau_{z+}})\mathbf{1}_{\left\{\tau_{z+}<\infty\right\}}\right)}{g(z)}=\frac{E\left(e^{-q\tau_{0+}}g(z+X_{\tau_{0+}})\mathbf{1}_{\left\{\tau_{0+}<\infty\right\}}\right)}{g(z)}
$$
is decreasing in $z\in(x_0,\infty)$, implying by (\ref{d24}) that
\begin{equation}\label{d25}
E_z\left(e^{-q\tau_{z+}}g(X_{\tau_{z+}})\mathbf{1}_{\left\{\tau_{z+}<\infty\right\}}\right)\le g(z)\;\;\mbox{for}\;\;z>x.
\end{equation}
Let $J_0=\tau_{x+}$ and for $n\ge1$,
$$
J_n=
\begin{cases}
\inf\{t>J_{n-1}:X_t>X_{J_{n-1}}\},&\;\;\mbox{if}\;\;J_{n-1}<\infty;\\
\infty,&\;\;\mbox{otherwise.}
\end{cases}
$$
Note that $\mathcal{L}(J_{n+1}-J_n\mid X_0=x,J_n<\infty)=\mathcal{L}(\tau_{0+}\mid X_0=0)$ and that $P(\tau_{0+}>0)=1$. It follows that $J_n\to\infty$ a.s. For (fixed) $y>x$, let $L_n=\min\{J_n,\tau_y\}$. Since the L\'{e}vy process $X$ either satisfies $\overline{\lim}_{t\to\infty}X_t=+\infty$ a.s. or $\lim_{t\to\infty}X_t=-\infty$ a.s., we have $J_n=\tau_y$ for large $n$ in the former case and $J_n=\infty$ for large $n$ in the latter case. In either case, $L_n=\min\{J_n,\tau_y\}=\tau_y$ for large $n$. As a consequence,
\begin{equation*}
e^{-qL_n}g(X_{L_n})\mathbf{1}_{\{L_n<\infty\}}\to e^{-q\tau_y}g(X_{\tau_y})\mathbf{1}_{\{\tau_y<\infty\}}\;\;\mbox{a.s.},
\end{equation*}
so that by Fatou's lemma,
\begin{equation}\label{d26}
\mathop{\underline{\lim}}_{n\to\infty}E_x\left(e^{-qL_n}g(X_{L_n})\mathbf{1}_{\{L_n<\infty\}}\right)\ge E_x\left(e^{-q\tau_y}g(X_{\tau_y})\mathbf{1}_{\{\tau_y<\infty\}}\right).
\end{equation}
By (\ref{d25}), it is readily shown that $E_x\left(e^{-qL_n}g(X_{L_n})\mathbf{1}_{\{L_n<\infty\}}\right)$ is decreasing in $n$, which together with (\ref{d26}) implies that
$$
E_x\left(e^{-q\tau_y}g(X_{\tau_y})\mathbf{1}_{\{\tau_y<\infty\}}\right)\le E_x\left(e^{-qL_0}g(X_{L_0})\mathbf{1}_{\{L_0<\infty\}}\right)\le g(x).
$$
The proof is complete.
\end{proof}

Let
\begin{equation}\label{d16}
u'=\inf\left\{x\in\mathbb{R}:\frac{E_x\left(e^{-q\tau_{x+}}g(X_{\tau_{x+}})\mathbf{1}_{\{\tau_{x+}<\infty\}}\right)}{g(x)}\le1\right\}.
\end{equation}

\begin{lem}\label{g5}
Let $g:\mathbb{R}\to[0,\infty)$ be nonconstant, increasing, logconcave and right-continuous. Define $u$ and $u'$ as in $(\ref{b2})$ and $(\ref{d16})$. If $0$ is irregular for $(0,\infty)$ for $X$, then $u'=u$.
\end{lem}
\begin{proof}
We first show that $u\ge u'$. It suffices to consider the case $u<\infty$. By Theorem \ref{t3.1}, for $x\ge u$,
$$
E_x\left(e^{-q\tau_{x+}}g(X_{\tau_{x+}})\mathbf{1}_{\{\tau_{x+}<\infty\}}\right)\le V(x)=g(x),
$$
so that $u'\le u$ by (\ref{d16}). To show $u\le u'$, suppose to the contrary that $u>u'$. Let $x$ be such that $u>x>u'(\ge x_0)$. If $u<\infty$, it follows from (\ref{d16}) and Lemma \ref{g4} that
\begin{align*}
g(x)&\ge E_x\left(e^{-q\tau_{u}}g(X_{\tau_{u}})\mathbf{1}_{\{\tau_{u}<\infty\}}\right)\\
&=V(x)\;\;\mbox{(by Theorem \ref{t3.1}(i))}\\
&>g(x)\;\;\mbox{(since $x<u$)},
\end{align*}
a contradiction. If $u=\infty$, it follows from (\ref{d16}) and Lemma \ref{g4} that
$$
g(x)\ge E_x\left(e^{-q\tau_{y}}g(X_{\tau_{y}})\mathbf{1}_{\{\tau_{y}<\infty\}}\right)\;\;\mbox{for all}\;\;y>x.
$$
By Theorem \ref{t3.1}(ii),
\begin{align*}
g(x)&\ge \lim_{y\to\infty}E_x\left(e^{-q\tau_{y}}g(X_{\tau_{y}})\mathbf{1}_{\{\tau_{y}<\infty\}}\right)\\
&=V(x)>g(x),
\end{align*}
a contradiction. This proves $u\le u'$ and completes the proof.
\end{proof}

\begin{proof}[\bf Proof of Theorem \ref{g3}]
(i) We have
\begin{align}
&E_u\left(e^{-q\tau_{u+}}g(X_{\tau_{u+}})\mathbf{1}_{\{\tau_{u+}<\infty\}}\right)\le V(u)=g(u),\label{a2}\\
&E_x\left(e^{-q\tau_{x+}}g(X_{\tau_{x+}})\mathbf{1}_{\{\tau_{x+}<\infty\}}\right)>g(x)\;\;\mbox{for}\;\;x_0<x<u,\label{a3}
\end{align}
where (\ref{a3}) is due to the definition of $u'$ and $u'=u$ (by Lemma \ref{g5}). Since both $g(x)$ and
$$
E_x\left(e^{-q\tau_{x+}}g(X_{\tau_{x+}})\mathbf{1}_{\{\tau_{x+}<\infty\}}\right)=E\left(e^{-q\tau_{0+}}g(x+X_{\tau_{0+}})\mathbf{1}_{\{\tau_{0+}<\infty\}}\right)
$$
are continuously increasing in $x>x_0$, (\ref{a2}) and (\ref{a3}) together imply that
\begin{equation}\label{a4}
g(u)=E_u\left(e^{-q\tau_{u+}}g(X_{\tau_{u+}})\mathbf{1}_{\{\tau_{u+}<\infty\}}\right)=E\left(e^{-q\tau_{0+}}g(u+X_{\tau_{0+}})\mathbf{1}_{\{\tau_{0+}<\infty\}}\right),
\end{equation}
which in turn implies that both $\tau_u$ and $\tau_{u+}$ are optimal stopping times.

For $x<u$,
\begin{align*}
V(x)&=E_x\left(e^{-q\tau_{u+}}g(X_{\tau_{u+}})\mathbf{1}_{\{\tau_{u+}<\infty\}}\right)\\
&=E_x\left(E_x\left(e^{-q\tau_{u+}}g(X_{\tau_{u+}})\mathbf{1}_{\{\tau_{u+}<\infty\}}\mid\mathcal{F}_{\tau_{x+}}\right)\right).
\end{align*}
On $\{\tau_{x+}<\infty\}$, since $X_{s+\tau_{x+}}-X_{\tau_{x+}}$ is independent of $\mathcal{F}_{\tau_{x+}}$ and has the same law as $X_s-X_0$ and since $\tau_{u+}$ is optimal, we have
$$
E_x\left(e^{-q(\tau_{u+}-\tau_{x+})}g(X_{\tau_{u+}})\mathbf{1}_{\{\tau_{u+}<\infty\}}\mid\mathcal{F}_{\tau_{x+}}\right)=V(X_{\tau_{x+}})\;\;\mbox{a.s.}
$$
It follows that
$$
V(x)=E_x\left(e^{-q\tau_{x+}}V(X_{\tau_{x+}})\mathbf{1}_{\{\tau_{x+}<\infty\}}\right).
$$
Taking $x=u-\varepsilon$ for $\varepsilon>0$ yields
\begin{align}
V(u-\varepsilon)&=E_{u-\varepsilon}\left(e^{-q\tau_{(u-\varepsilon)+}}V(X_{\tau_{(u-\varepsilon)+}})\mathbf{1}_{\{\tau_{(u-\varepsilon)+}<\infty\}}\right)\notag\\
&=E\left(e^{-q\tau_{0+}}V(u-\varepsilon+X_{\tau_{0+}})\mathbf{1}_{\{\tau_{0+}<\infty\}}\right).\label{a5}
\end{align}
By (\ref{a4}) and (\ref{a5}),
\begin{align*}
\frac{V(u)-V(u-\varepsilon)}{\varepsilon}&=\frac{1}{\varepsilon}E\left[e^{-q\tau_{0+}}\left(g(u+X_{\tau_{0+}})-V(u-\varepsilon+X_{\tau_{0+}})\right)\mathbf{1}_{\{\tau_{0+}<\infty\}}\right]\\
&=A(\varepsilon)-B(\varepsilon),
\end{align*}
where
\begin{align*}
A(\varepsilon)&=\frac{1}{\varepsilon}E\left[e^{-q\tau_{0+}}\left(g(u+X_{\tau_{0+}})-g(u-\varepsilon+X_{\tau_{0+}})\right)\mathbf{1}_{\{\tau_{0+}<\infty\}}\right],\\
B(\varepsilon)&=\frac{1}{\varepsilon}E\left[e^{-q\tau_{0+}}\left(V(u-\varepsilon+X_{\tau_{0+}})-g(u-\varepsilon+X_{\tau_{0+}})\right)\mathbf{1}_{\{\tau_{0+}<\infty\}}\right].
\end{align*}

To show
\begin{equation}\label{a6}
\lim_{\varepsilon\downarrow0}A(\varepsilon)=E\left(e^{-q\tau_{0+}}g'((u+X_{\tau_{0+}})-)\mathbf{1}_{\{\tau_{0+}<\infty\}}\right),
\end{equation}
we have
\begin{align}
g(u+X_{\tau_{0+}})-g(u-\varepsilon+X_{\tau_{0+}})&=g(u-\varepsilon+X_{\tau_{0+}})\left[e^{h(u+X_{\tau_{0+}})-h(u-\varepsilon+X_{\tau_{0+}})}-1\right]\notag\\
&\le g(u-\varepsilon+X_{\tau_{0+}})\left[e^{\varepsilon h'((u-\varepsilon+X_{\tau_{0+}})+)}-1\right]\label{a8}\\
&=g(u-\varepsilon+X_{\tau_{0+}})e^{\theta\varepsilon h'((u-\varepsilon+X_{\tau_{0+}})+)}\varepsilon h'((u-\varepsilon+X_{\tau_{0+}})+)\notag
\end{align}
for some $\theta\in(0,1)$, where the inequality is due to the concavity of $h(x)$ and the last equality follows from the mean value theorem applied to the function $e^x$. So, for any (fixed) $v\in(x_0,u)$ and for $0<\varepsilon<u-v$, on $\{\tau_{0+}<\infty\}$,
\begin{align*}
\frac{1}{\varepsilon}\left[g(u+X_{\tau_{0+}})-g(u-\varepsilon+X_{\tau_{0+}})\right]&\le g(u-\varepsilon+X_{\tau_{0+}})e^{\theta\varepsilon h'((u-\varepsilon+X_{\tau_{0+}})+)}h'((u-\varepsilon+X_{\tau_{0+}})+)\\
&\le g(u+X_{\tau_{0+}})e^{\varepsilon h'(v+)}h'((u-\varepsilon+X_{\tau_{0+}})+)\\
&\le g(u+X_{\tau_{0+}})e^{(u-v)h'(v+)}h'(v+).
\end{align*}
Since on $\{\tau_{0+}<\infty\}$
$$
\lim_{\varepsilon\downarrow0}g(u+X_{\tau_{0+}})e^{\varepsilon h'(v+)}h'((u-\varepsilon+X_{\tau_{0+}})+)=g(u+X_{\tau_{0+}})h'((u+X_{\tau_{0+}})-)\;\;\mbox{a.s.},
$$
we have
\begin{align}
\overline{\lim_{\varepsilon\downarrow0}}A(\varepsilon)&\le\lim_{\varepsilon\downarrow0}E\left[e^{-q\tau_{0+}}g(u+X_{\tau_{0+}})e^{\varepsilon h'(v+)}h'((u-\varepsilon+X_{\tau_{0+}})+)\mathbf{1}_{\{\tau_{0+}<\infty\}}\right]\notag\\
&=E\left[e^{-q\tau_{0+}}g(u+X_{\tau_{0+}})h'((u+X_{\tau_{0+}})-)\mathbf{1}_{\{\tau_{0+}<\infty\}}\right]\notag\\
&=E\left(e^{-q\tau_{0+}}g'((u+X_{\tau_{0+}})-)\mathbf{1}_{\{\tau_{0+}<\infty\}}\right),\label{a7}
\end{align}
by the dominated convergence theorem together with the fact (\textit{cf.} (\ref{a4})) that
$$
E\left(e^{-q\tau_{0+}}g(u+X_{\tau_{0+}})\mathbf{1}_{\{\tau_{0+}<\infty\}}\right)=g(u)<\infty.
$$
Instead of the upper bound in (\ref{a8}), we can use the lower bound
$$
g(u+X_{\tau_{0+}})-g(u-\varepsilon+X_{\tau_{0+}})\ge g(u-\varepsilon+X_{\tau_{0+}})\left[e^{\varepsilon h'((u+X_{\tau_{0+}})-)}-1\right]
$$
to derive in a similar way
$$
\mathop{\underline{\lim}}_{\varepsilon\downarrow0}A(\varepsilon)\ge E\left(e^{-q\tau_{0+}}g'((u+X_{\tau_{0+}})-)\mathbf{1}_{\{\tau_{0+}<\infty\}}\right),
$$
which together with (\ref{a7}) establishes (\ref{a6}).

To show $\lim_{\varepsilon\downarrow0}B(\varepsilon)=0$, note that $V(x)=g(x)$ for $x\ge u$ and $g(x)<V(x)\le g(u)$ for $x<u$. So
\begin{align*}
B(\varepsilon)&=\frac{1}{\varepsilon}E\left[e^{-q\tau_{0+}}\left(V(u-\varepsilon+X_{\tau_{0+}})-g(u-\varepsilon+X_{\tau_{0+}})\right)\mathbf{1}_{\{\tau_{0+}<\infty,\;X_{\tau_{0+}}<\varepsilon\}}\right]\\
&\le\frac{1}{\varepsilon}E\left[\left(g(u)-g(u-\varepsilon)\right)\mathbf{1}_{\{\tau_{0+}<\infty,\;X_{\tau_{0+}}<\varepsilon\}}\right]\\
&=\frac{1}{\varepsilon}\left(g(u)-g(u-\varepsilon)\right)P(\tau_{0+}<\infty,X_{\tau_{0+}}<\varepsilon),
\end{align*}
implying that
\begin{equation}\label{a9}
\overline{\lim_{\varepsilon\downarrow0}}B(\varepsilon)\le g'(u-)P(\tau_{0+}<\infty,X_{\tau_{0+}}=0)=0,
\end{equation}
(since $0$ is irregular for $(0,\infty)$ and $X$ does not creep upwards). Combining (\ref{a6}) and (\ref{a9}) yields
$$
V'(u-)=\lim_{\varepsilon\downarrow0}\frac{V(u)-V(u-\varepsilon)}{\varepsilon}=E\left(e^{-q\tau_{0+}}g'((u+X_{\tau_{0+}})-)\mathbf{1}_{\{\tau_{0+}<\infty\}}\right).
$$

(ii) We have by part (i)
\begin{align*}
V'(u-)&=E\left(e^{-q\tau_{0+}}g'((u+X_{\tau_{0+}})-)\mathbf{1}_{\{\tau_{0+}<\infty\}}\right)\\
&=E\left(e^{-q\tau_{0+}}h'((u+X_{\tau_{0+}})-)g(u+X_{\tau_{0+}})\mathbf{1}_{\{\tau_{0+}<\infty\}}\right)\\
&\le E\left(e^{-q\tau_{0+}}h'(u+)g(u+X_{\tau_{0+}})\mathbf{1}_{\{\tau_{0+}<\infty\}}\right)\\
&=h'(u+)g(u)\;\;(\mbox{by (\ref{a4})})\\
&=g'(u+)=V'(u+),
\end{align*}
where the inequality is due to the concavity of $h(x)$. This inequality is an equality if and only if
$$
P(h'((u+X_{\tau_{0+}})-)=h'(u+)\mid\tau_{0+}<\infty)=1,
$$
which is equivalent to $h'((u+\zeta)-)=h'(u+)$.

(iii) Condition (\ref{a1}) implies that $h'(x-)=h'(x+)=h'(u+)$ for $u<x<u+\zeta$, which in turn implies that
\begin{equation}\label{a10}
g(x)=g(u)e^{h'(u+)(x-u)}\;\;\mbox{for}\;\;u\le x\le u+\zeta.
\end{equation}
By (\ref{a4}) and (\ref{a10}),
\begin{align}
g(u)&=E\left(e^{-q\tau_{0+}}g(u+X_{\tau_{0+}})\mathbf{1}_{\{\tau_{0+}<\infty\}}\right)\notag\\
&=E\left(e^{-q\tau_{0+}}g(u)e^{h'(u+)X_{\tau_{0+}}}\mathbf{1}_{\{\tau_{0+}<\infty\}}\right)\notag\\
&=g(u)E\left(e^{-q\tau_{0+}}e^{h'(u+)X_{\tau_{0+}}}\mathbf{1}_{\{\tau_{0+}<\infty\}}\right).\label{a11}
\end{align}
Let $\widetilde{V}(x)=g(u)e^{h'(u+)(x-u)}$ for $x\in\mathbb{R}$, so that
\begin{equation}\label{a12}
\widetilde{V}(x)=g(x)\;\;\mbox{for}\;\;u\le x\le u+\zeta.
\end{equation}
Note that for all $y\in\mathbb{R}$,
\begin{align}
E_y\left(e^{-q\tau_{y+}}\widetilde{V}(X_{\tau_{y+}})\mathbf{1}_{\{\tau_{y+}<\infty\}}\right)&=E\left(e^{-q\tau_{0+}}\widetilde{V}(y+X_{\tau_{0+}})\mathbf{1}_{\{\tau_{0+}<\infty\}}\right)\notag\\
&=g(u)e^{h'(u+)(y-u)}E\left(e^{-q\tau_{0+}}e^{h'(u+)X_{\tau_{0+}}}\mathbf{1}_{\{\tau_{0+}<\infty\}}\right)\notag\\
&=g(u)e^{h'(u+)(y-u)}\notag\\
&=\widetilde{V}(y),\label{a13}
\end{align}
where the second-to-last equality follows from (\ref{a11}). Fix an (arbitrary) $x<u$. We want to show
\begin{equation}\label{a14}
V(x)=\widetilde{V}(x)=g(u)e^{h'(u+)(x-u)}.
\end{equation}
Let $J_0=\tau_{x+}$ and for $n\ge1$,
$$
J_n=
\begin{cases}
\inf\{t>J_{n-1}:X_t>X_{J_{n-1}}\},&\;\;\mbox{if}\;\;J_{n-1}<\infty;\\
\infty,&\;\;\mbox{otherwise.}
\end{cases}
$$
Let $L_n=\min\{J_n,\tau_{u+}\}$. Note if $L_n=J_n<\tau_{u+}$, then $J_{n+1}\le\tau_{u+}$ so that $L_{n+1}=J_{n+1}\le\tau_{u+}$. Following the argument in the proof of Lemma \ref{g4}, we have $L_n=\tau_{u+}$ for large $n$, so that
\begin{equation}\label{a15}
\lim_{n\to\infty}e^{-qL_n}\widetilde{V}(X_{L_n})\mathbf{1}_{\{L_n<\infty\}}=e^{-q\tau_{u+}}\widetilde{V}(X_{\tau_{u+}})\mathbf{1}_{\{\tau_{u+}<\infty\}}\;\;\mbox{a.s.}
\end{equation}
Note that $X_{L_n}\le u$ if $L_n<\tau_{u+}$, implying that
$$
e^{-qL_n}\widetilde{V}(X_{L_n})\mathbf{1}_{\{L_n<\infty\}}\le \widetilde{V}(u)+e^{-q\tau_{u+}}\widetilde{V}(X_{\tau_{u+}})\mathbf{1}_{\{\tau_{u+}<\infty\}}.
$$
By (\ref{a15}) and the dominated convergence theorem,
\begin{align}
\lim_{n\to\infty}E_x\left(e^{-qL_n}\widetilde{V}(X_{L_n})\mathbf{1}_{\{L_n<\infty\}}\right)&=E_x\left(e^{-q\tau_{u+}}\widetilde{V}(X_{\tau_{u+}})\mathbf{1}_{\{\tau_{u+}<\infty\}}\right)\notag\\
&=E_x\left(e^{-q\tau_{u+}}g(X_{\tau_{u+}})\mathbf{1}_{\{\tau_{u+}<\infty\}}\right)\;\;\mbox{(by (\ref{a12}))}\notag\\
&=V(x),\label{a16}
\end{align}
where the last equality is due to the optimality of $\tau_{u+}$. On the other hand, it follows from (\ref{a13}) that
$$
E_x\left(e^{-qL_{n+1}}\widetilde{V}(X_{L_{n+1}})\mathbf{1}_{\{L_{n+1}<\infty\}}\right)=E_x\left(e^{-qL_n}\widetilde{V}(X_{L_n})\mathbf{1}_{\{L_n<\infty\}}\right).
$$
We have by (\ref{a16})
\begin{align*}
V(x)&=\lim_{n\to\infty}E_x\left(e^{-qL_n}\widetilde{V}(X_{L_n})\mathbf{1}_{\{L_n<\infty\}}\right)\\
&=E_x\left(e^{-qL_0}\widetilde{V}(X_{L_0})\mathbf{1}_{\{L_0<\infty\}}\right)\\
&=E_x\left(e^{-q\tau_{x+}}\widetilde{V}(X_{\tau_{x+}})\mathbf{1}_{\{\tau_{x+}<\infty\}}\right)\\
&=\widetilde{V}(x)\;\;\mbox{(by (\ref{a13}))},
\end{align*}
establishing (\ref{a14}). The proof is complete.
\end{proof}

\begin{rem}
Theorems \ref{g6} and \ref{g3} assume that $x_0<u<\infty$, which makes it unnecessary to require the right-continuity of $g$ at $x_0$. Let $g(x)\ge0$ be increasing and logconcave. We say the $g(x)$ is degenerate if $\log g(x)$ is linear for $x$ in some interval. Note that $g$ is degenerate if condition $(\ref{a1})$ holds. For a nondegenerate $g(x)$, $q\ge0$ and L\'{e}vy process $X$, suppose the optimal threshold $u$ $($given in $(\ref{b2}))$ satisfies $x_0<u<\infty$ and $g$ is differentiable at $u$. Then by Theorems \ref{g6} and \ref{g3}, the principle of smooth fit holds if and only if $0$ is regular for $(0,\infty)$ for $X$.
\end{rem}

\begin{rem}
If $g$ is degenerate, the principle of smooth fit may hold even when $0$ is irregular for $(0,\infty)$ for $X$. As an example, consider the L\'{e}vy process $X_t=-t+N_t$, $t\ge0$, where $N_t$ is a Poisson process with rate $\mu>0$. Then $0$ is irregular for $(0,\infty)$. Define $\phi(\lambda):=E(e^{\lambda X_1})=e^{\mu e^{\lambda}-\lambda-\mu}$ for $\lambda\ge0$, which is convex. Let
$$
\lambda'=\sup\{\lambda\ge0:\phi(\lambda)=e^{q}\},
$$
which is positive for $q>0$. For $q=0$, we assume $\mu<1$ so that $E(X_1)<0$ and $\lambda'>0$. Since $E(e^{-q+\lambda'X_1})=1$, we have $E(e^{-qt+\lambda'X_t})=\left[E(e^{-q+\lambda'X_1})\right]^{t}=1$ for $t>0$, so that $\{e^{-qt+\lambda'X_t}\}_{t\ge0}$ is a martingale. Consequently,
\begin{equation}\label{d21}
E\left(e^{-q\tau_{0+}+\lambda'X_{\tau_{0+}}}\mathbf{1}_{\{\tau_{0+}<\infty\}}\right)=1.
\end{equation}
Note that $\zeta=\inf\{x:P(X_{\tau_{0+}}>x\mid\tau_{0+}<\infty)=0\}=1$. Let $g(x)=e^{h(x)}$ where
$$
h(x)=
\begin{cases}
\lambda'x-x^2,&\;\;\mbox{for}\;\;x\le0,\\
\lambda'x,&\;\;\mbox{for}\;\;0<x\le1,\\
\lambda'(2-x^{-1}),&\;\;\mbox{for}\;\;x>1,
\end{cases}
$$
which is increasing, concave and continuously differentiable. We have by $(\ref{d21})$
\begin{equation}\label{a17}
E\left(e^{-q\tau_{0+}}g(X_{\tau_{0+}})\mathbf{1}_{\{\tau_{0+}<\infty\}}\right)=E\left(e^{-q\tau_{0+}+\lambda'X_{\tau_{0+}}}\mathbf{1}_{\{\tau_{0+}<\infty\}}\right)=1=g(0),
\end{equation}
and for $x<0$,
\begin{align}
\frac{E_x\left(e^{-q\tau_{x+}}g(X_{\tau_{x+}})\mathbf{1}_{\{\tau_{x+}<\infty\}}\right)}{g(x)}&=\frac{E\left(e^{-q\tau_{0+}}g(x+X_{\tau_{0+}})\mathbf{1}_{\{\tau_{0+}<\infty\}}\right)}{e^{\lambda'x-x^2}}\notag\\
&=\frac{E\left[\mathrm{exp}\left(-q\tau_{0+}+\lambda'(x+X_{\tau_{0+}})-(\min\{x+X_{\tau_{0+}},0\})^2\right)\mathbf{1}_{\{\tau_{0+}<\infty\}}\right]}{e^{\lambda'x-x^2}}\notag\\
&=E\left[\mathrm{exp}\left(-q\tau_{0+}+\lambda'X_{\tau_{0+}}+x^2-(\min\{x+X_{\tau_{0+}},0\})^2\right)\mathbf{1}_{\{\tau_{0+}<\infty\}}\right]\notag\\
&>E\left[\mathrm{exp}(-q\tau_{0+}+\lambda'X_{\tau_{0+}})\mathbf{1}_{\{\tau_{0+}<\infty\}}\right]=1,\label{a18}
\end{align}
where the inequality follows from the fact that
$$
P(x^2-(\min\{x+X_{\tau_{0+}},0\})^2>0\mid\tau_{0+}<\infty)=P(X_{\tau_{0+}}>0\mid\tau_{0+}<\infty)=1\;\;\mbox{for}\;\;x<0.
$$
If follows from $(\ref{d16})$, $(\ref{a17})$--$(\ref{a18})$ and Lemma \ref{g5} that $u=u'=0$. By Theorem \ref{g3}{\upshape(iii)} $($or direct calculations$)$, we have $V(x)=e^{\lambda'x}$ for $x<0$. More generally, for any increasing and logconcave function $g(x)$ satisfying $0\le g(x)<e^{\lambda'x}$ for $x<0$, $g(x)=e^{\lambda'x}$ for $0\le x\le1$, and $g(x)\le e^{\lambda'x}$ for $x>1$, it can be shown that $V(x)=e^{\lambda'x}$ for $x<0$ and $V(x)=g(x)$ for $x\ge0$. For example, let
$$
g(x)=e^{ax}\mathbf{1}_{(-\infty,0)}(x)+e^{\lambda'x}\mathbf{1}_{[0,1]}(x)+e^{b(x-1)+\lambda'}\mathbf{1}_{(1,\infty)}(x)
$$
with $a>\lambda'>b\ge0$, which is not differentiable at $x=0,1$. The value function is given by
$$
V(x)=e^{\lambda'x}\mathbf{1}_{(-\infty,1]}(x)+e^{b(x-1)+\lambda'}\mathbf{1}_{(1,\infty)}(x).
$$
\end{rem}

\section{Concluding remarks and conditions for $u<\infty$}
\hspace*{18pt}The optimal stopping problem (\ref{e1.1}) involves the reward function $g$ and the underlying process $\{X_t\}$ as well as the discount rate $q\ge0$. Motivated by well-known results in the literature, we explored the close connection between increasing and logconcave reward functions and optimal stopping times of threshold type. Specifically in this paper, $g$ is assumed to be nonnegative, nonconstant, increasing, logconcave and right-continuous while $\{X_t\}$ is either a random walk in discrete time or a L\'{e}vy process in continuous time. We showed that there exists a unique threshold $u\in[-\infty,\infty]$ such that
\begin{itemize}
\item[(i)] the value function $V(x)>g(x)$ for $x<u$ and $V(x)=g(x)$ for $x\ge u$;
\item[(ii)] if $-\infty\le u<\infty$, then $\tau_u=\inf\{t\ge0:X_t\ge u\}$ is optimal;
\item[(iii)] if $u=\infty$, the stopping region $\{x:V(x)=g(x)\}=[u,\infty)=\emptyset$ and it is never optimal to stop since more profit can always be made by stopping at a later time. (However, to never stop yields a zero reward.)
\end{itemize}

The work of Alili and Kyprianou \cite{ref18} makes use of a fluctuation identity to give insight into the importance of the role played by the regularity of the paths of $\{X_t\}$ in the solution for the American put optimal stopping problem. Building on it, we investigated the principle of smooth fit more generally when $g$ is increasing and logconcave. We obtained necessary and sufficient conditions for the smooth fit principle to hold.

Finally we discuss conditions for the threshold $u<\infty$, which is given in (\ref{g1}) and (\ref{b2}) for the discrete- and continuous-time cases, respectively. A simple sufficient condition for $u<\infty$ is
\begin{equation}\label{c3}
E\left[e^{(\beta+\delta)X_1}\right]\le e^{q}\;\;\mbox{for some}\;\;\delta>0,
\end{equation}
where $\beta$ is given in (\ref{e1}). Note that (\ref{c3}) implies that $E(e^{\lambda X_1})<\infty$ is convex in $\lambda\in[0,\beta+\delta]$, so that
\begin{equation}\label{c4}
E(e^{\lambda X_1})\le\max\left\{E\left[e^{(\beta+\delta)X_1}\right],1\right\}\le e^{q},\;\;\mbox{for}\;\;\lambda\in[0,\beta+\delta].
\end{equation}
Under condition (\ref{c3}), we claim that $u\le x$ for any $x$ with $c:=h'(x-)\le\beta+\delta$ where $h(x)=\log g(x)$. To see this, let $\tilde{h}(y):=h(x)+c(y-x)\ge h(y)$ for $y\in\mathbb{R}$, so that $g(y)\le\tilde{g}(y):=e^{\tilde{h}(y)}$ for $y\in\mathbb{R}$. Then in the discrete-time case,
\begin{align}
E_x\left(e^{-qT_x}g(X_{T_x})\mathbf{1}_{\{T_x<\infty\}}\right)&\le E_x\left(e^{-qT_x}\tilde{g}(X_{T_x})\mathbf{1}_{\{T_x<\infty\}}\right)\notag\\
&=E_x\left[e^{-qT_x}e^{h(x)+c(X_{T_x}-x)}\mathbf{1}_{\{T_x<\infty\}}\right]\notag\\
&=g(x)E_x\left[e^{-qT_x+c(X_{T_x}-x)}\mathbf{1}_{\{T_x<\infty\}}\right]\notag\\
&=g(x)E\left(e^{-qT_0+cX_{T_0}}\mathbf{1}_{\{T_0<\infty\}}\right)\le g(x),\label{c5}
\end{align}
where the last inequality is due to the fact that $\left\{e^{-qt+cX_t}:t=0,1,\dots\right\}$ is a supermartingale. (Note that by (\ref{c4}), $E(e^{-q+cX_1})\le1$ with $0\le c=h'(x-)\le\beta+\delta$.) We have by (\ref{g1}) and (\ref{c5}) that $u\le x$. In the continuous-time case, we have by (\ref{c4}) that for $t=2^{-\ell}$, $\ell=1,2,\dots$,
$$
E\left(e^{-qt+\lambda X_t}\right)=\left[E\left(e^{-q+\lambda X_1}\right)\right]^{t}\le1,\;\;\mbox{for}\;\;\lambda\in[0,\beta+\delta],
$$
which implies that $u^{(\ell)}\le x$ for all $\ell=1,2,\dots$. Thus $u=\lim_{\ell\to\infty}u^{(\ell)}\le x$. This proves that (\ref{c3}) is sufficient for $u<\infty$.

On the other hand, suppose $h'(x-)>\beta=\lim_{y\to\infty}h'(y-)$ for all $x\in\mathbb{R}$ (which implies that $h(x)$ is not linear in $x\in(a,\infty)$ for any $a\in\mathbb{R}$). If $\max\{q,\beta\}>0$, then a sufficient condition for $u=\infty$ is
\begin{equation}\label{c6}
E\left(e^{\beta X_1}\right)\ge e^{q}.
\end{equation}
To see this, note that the discrete-time process $\left\{e^{-qt+\beta X_t}:t=0,1,\dots\right\}$ is a submartingale. Since $\max\{q,\beta\}>0$, it is readily shown that
\begin{equation}\label{c9}
1\le\lim_{n\to\infty}E\left(e^{-q(T_0\wedge n)+\beta X_{T_0\wedge n}}\right)=E\left(e^{-qT_0+\beta X_{T_0}}\mathbf{1}_{\{T_0<\infty\}}\right).
\end{equation}
For any $x$ with $g(x)>0$, we claim that
\begin{equation}\label{c7}
g(x)<E_x\left(e^{-qT_x}g(X_{T_x})\mathbf{1}_{\{T_x<\infty\}}\right).
\end{equation}
Let $h^*(y)=h(x)+\beta(y-x)$ and $g^*(y)=e^{h^*(y)}=g(x)e^{\beta(y-x)}$ for $y\in\mathbb{R}$. Recall that by assumption, $h'(y-)>\beta=(h^*)'(y)$ for all $y\in\mathbb{R}$, so $g^*(x)=g(x)$ and $g^*(y)<g(y)$ for $y>x$. Since $P_x(X_{T_x}>x,T_x<\infty)\ge P(\xi>0)>0$, we have
\begin{align*}
E_x\left(e^{-qT_x}g(X_{T_x})\mathbf{1}_{\{T_x<\infty\}}\right)&>E_x\left(e^{-qT_x}g^*(X_{T_x})\mathbf{1}_{\{T_x<\infty\}}\right)\\
&=g(x)E\left(e^{-qT_0+\beta X_{T_0}}\mathbf{1}_{\{T_0<\infty\}}\right)\\
&\ge g(x)\;\;\mbox{(by (\ref{c9}))}.
\end{align*}
This proves (\ref{c7}), implying that $u=\infty$ for the discrete-time case. The continuous-time case can be treated similarly. By (\ref{c3}) and (\ref{c6}) together, we can characterize $u<\infty$ as follows. Assume again that $h'(x-)>\beta$ for all $x\in\mathbb{R}$ and that $\max\{q,\beta\}>0$. Suppose $E\left[e^{(\beta+\delta)X_1}\right]<\infty$ for some $\delta>0$. Then $E(e^{\beta X_1})<e^{q}$ implies $E\left[e^{(\beta+\delta')X_1}\right]<e^{q}$ for sufficiently small $\delta'>0$, which in turn implies $u<\infty$ by (\ref{c3}). On the other hand, by (\ref{c6}), $E(e^{\beta X_1})\ge e^{q}$ implies $u=\infty$. Thus we have $u<\infty$ if and only if $E(e^{\beta X_1})<e^{q}$ provided that $E\left[e^{(\beta+\delta)X_1}\right]<\infty$ for some $\delta>0$. When $E(e^{\beta X_1})<e^{q}$ and $E\left[e^{(\beta+\delta)X_1}\right]=\infty$ for all $\delta>0$, it is unclear how to find general conditions for $u<\infty$. In the special case $g(x)=(x^+)^{\nu}$ with $\nu>0$ (for which $\beta=0$), Novikov and Shiryaev \cite{ref9} showed that $u<\infty$ if {\em either} $q>0$, $E[(X_1^+)^{\nu}]<\infty$ {\em or} $q=0$, $E(X_1)<0$, $E[(X_1^+)^{\nu+1}]<\infty$.


\section{Proofs of Lemmas \ref{lc}--\ref{le}}

\begin{proof}[\bf Proof of Lemma \ref{lc}]
(i) For $a\in[-\infty,y)$, let $J_0=T_a$ and for $n\ge1$,
$$
J_n=
\begin{cases}
\inf\{j>J_{n-1}:X_j\ge X_{J_{n-1}}\},\;\;&\mbox{if}\;J_{n-1}<\infty;\\
\infty,\;\;&\mbox{otherwise.}
\end{cases}
$$
Here starting at $X_{T_a}=X_{J_0}$ (if $T_a<\infty$), the (finite) $J_1,J_2,\dots$ are the weak ascending ladder epochs and $X_{J_1},X_{J_2},\dots$ are the corresponding weak ascending ladder heights. Let $L_n=\min\{J_n,T_y\}$ for $n\ge0$. It is well known (\textit{cf.} Theorem 8.2.5 of Chung \cite{ref16}) that the random walk $\{X_n\}$ either satisfies $\overline{\lim}_{n\to\infty}X_n=+\infty$ a.s. or $\lim_{n\to\infty}X_n=-\infty$ a.s. If $\overline{\lim}_{n\to\infty}X_n=+\infty$ a.s., then a.s. $T_y<\infty$ and $0<J_0<J_1<J_{2}<\cdots$ are all finite, so that $L_n=\min\{J_n,T_y\}=T_y$ for large $n$. If $\lim_{n\to\infty}X_n=-\infty$ a.s., then a.s. there exists a finite $n'\ge0$ such that $J_n=\infty$ for all $n\ge n'$, implying that $L_n=\min\{J_n,T_y\}=T_y$ for all $n\ge n'$. Thus, regardless of whether $\overline{\lim}_{n\to\infty}X_n=+\infty$ a.s. or $\lim_{n\to\infty}X_n=-\infty$ a.s., we have $L_n=T_y$ for large $n$ a.s. As a consequence,
\begin{equation}\label{v10}
e^{-qL_n}g\left(X_{L_n}\right)\mathbf{1}_{\left\{L_n<\infty\right\}}\to e^{-qT_y}g(X_{T_y})\mathbf{1}_{\left\{T_y<\infty\right\}}\;\;\mbox{a.s.}
\end{equation}
It follows that
\begin{equation}\label{v11}
E_x\left(e^{-qL_n}g\left(X_{L_n}\right)\mathbf{1}_{\left\{L_n<\infty\right\}}\right)\to E_x\left(e^{-qT_y}g(X_{T_y})\mathbf{1}_{\left\{T_y<\infty\right\}}\right).
\end{equation}
More precisely, if $E_x\left(e^{-qT_y}g(X_{T_y})\mathbf{1}_{\left\{T_y<\infty\right\}}\right)=\infty$, then (\ref{v11}) follows from (\ref{v10}) and Fatou's lemma. If $E_x\left(e^{-qT_y}g(X_{T_y})\mathbf{1}_{\left\{T_y<\infty\right\}}\right)<\infty$, then (\ref{v11}) follows from (\ref{v10}) and the dominated convergence theorem upon observing
$$
e^{-qL_n}g(X_{L_n})\mathbf{1}_{\{L_n<\infty\}}\le \max\{g(y),e^{-qT_y}g(X_{T_y})\mathbf{1}_{\{T_y<\infty\}}\}.
$$

We now prove that for $n\ge1$,
\begin{equation}\label{v12}
E_x\left(e^{-qL_n}g\left(X_{L_n}\right)\mathbf{1}_{\left\{L_n<\infty\right\}}\right)\ge E_x\left(e^{-qL_{n-1}}g(X_{L_{n-1}})\mathbf{1}_{\left\{L_{n-1}<\infty\right\}}\right).
\end{equation}
Note that
\begin{align}
E_x\left(e^{-qL_n}g\left(X_{L_n}\right)\mathbf{1}_{\left\{L_n<\infty\right\}}\right)&=E_x\left(e^{-qL_{n-1}}\mathbf{1}_{\left\{L_{n-1}<\infty\right\}}e^{-q(L_n-L_{n-1})}g\left(X_{L_n}\right)\mathbf{1}_{\left\{L_n<\infty\right\}}\right)\notag\\
&=E_x\left[e^{-qL_{n-1}}\mathbf{1}_{\left\{L_{n-1}<\infty\right\}}E_x\left(e^{-q(L_n-L_{n-1})}g\left(X_{L_n}\right)\mathbf{1}_{\left\{L_n<\infty\right\}}\mid X_{L_{n-1}},L_{n-1}\right)\right]\label{e2.4}.
\end{align}
For (integer) $\ell<\infty$ and $x'<y$, we have
$$
\mathcal{L}(L_n-L_{n-1},X_{L_n}\mathbf{1}_{\{L_n<\infty\}}\mid X_{L_{n-1}}=x',L_{n-1}=\ell)=\mathcal{L}(T_{x'},X_{T_{x'}}\mathbf{1}_{\{T_{x'}<\infty\}}\mid X_0=x'),
$$
so that
\begin{align*}
\frac{E(e^{-q(L_n-L_{n-1})}g(X_{L_n})\mathbf{1}_{\{L_n<\infty\}}\mid X_{L_{n-1}}=x',L_{n-1}=\ell)}{g(x')}&=\frac{E_{x'}(e^{-qT_{x'}}g(X_{T_{x'}})\mathbf{1}_{\{T_{x'}<\infty\}})}{g(x')}\\
&\ge\frac{E_y(e^{-qT_y}g(X_{T_y})\mathbf{1}_{\{T_y<\infty\}})}{g(y)}\ge1,
\end{align*}
where the first inequality follows from (\ref{e2.18}). So on the event $A_{n,y}=\{L_{n-1}<\infty,X_{L_{n-1}}<y\}$, we have
\begin{equation}\label{e4.5}
E\left(e^{-q(L_n-L_{n-1})}g(X_{L_n})\mathbf{1}_{\left\{L_n<\infty\right\}}\mid X_{L_{n-1}},L_{n-1}\right)\ge g(X_{L_{n-1}})\mathbf{1}_{\left\{L_{n-1}<\infty\right\}}.
\end{equation}
It is easily seen that on $\Omega\backslash A_{n,y}$,
\begin{equation}\label{e4.6}
E\left(e^{-q(L_n-L_{n-1})}g(X_{L_n})\mathbf{1}_{\left\{L_n<\infty\right\}}\mid X_{L_{n-1}},L_{n-1}\right)= g(X_{L_{n-1}})\mathbf{1}_{\left\{L_{n-1}<\infty\right\}}.
\end{equation}
By (\ref{e2.4})--(\ref{e4.6}), (\ref{v12}) follows. Since $T_a=J_0=\min\{J_0,T_y\}=L_0$, we have by (\ref{v11}) and (\ref{v12})
\begin{align*}
E_x\left(e^{-qT_a}g(X_{T_a})\mathbf{1}_{\left\{T_a<\infty\right\}}\right)&=E_x\left(e^{-qL_0}g(X_{L_0})\mathbf{1}_{\left\{L_0<\infty\right\}}\right)\\
&\le E_x\left(e^{-qL_n}g(X_{L_n})\mathbf{1}_{\left\{L_n<\infty\right\}}\right)\\
&\to E_x\left(e^{-qT_y}g(X_{T_y})\mathbf{1}_{\left\{T_y<\infty\right\}}\right)\;\mbox{as}\;n\to\infty.
\end{align*}
This proves the desired inequality in Lemma \ref{lc}(i).

(ii) For $x<y$, we have by (\ref{e2.18}) that
$$
\frac{E_x\left(e^{-qT_x}g(X_{T_x})\mathbf{1}_{\left\{T_x<\infty\right\}}\right)}{g(x)}\ge \frac{E_y\left(e^{-qT_y}g(X_{T_y})\mathbf{1}_{\left\{T_y<\infty\right\}}\right)}{g(y)}\ge1,
$$
from which it follows that
\begin{align*}
g(x)&\le E_x\left(e^{-qT_x}g(X_{T_x})\mathbf{1}_{\left\{T_x<\infty\right\}}\right)\\
&\le E_x\left(e^{-qT_y}g(X_{T_y})\mathbf{1}_{\left\{T_y<\infty\right\}}\right)\;\;\mbox{(by Lemma \ref{lc}(i))}\\
&=E_x\left(e^{-q\tau_y}g(X_{\tau_y})\mathbf{1}_{\left\{\tau_y<\infty\right\}}\right).
\end{align*}
This completes the proof.
\end{proof}

\begin{proof}[\bf Proof of Lemma \ref{ld}]
Part (i) can be established along the lines of the proof of Lemma \ref{lc}(i) with minor changes, while part (ii) follows immediately from the assumption of the lemma and part (i). We leave out the details.
\end{proof}

\begin{proof}[\bf Proof of Lemma \ref{la}]
We first consider the special (trivial) case that $q=0$ and $E(\xi)\ge0$. Since $x_0\le u<\infty$, we have by (\ref{g1}) that
\begin{equation}\label{v13}
0<g(v)\ge E_{v}\left(g(X_{T_v})\mathbf{1}_{\{T_v<\infty\}}\right),\;\;\mbox{for}\;\; u<v<\infty.
\end{equation}
In view of $E(\xi)\ge0$ and $P(\xi>0)>0$, we have $T_v<\infty$ and $v\le X_{T_v}$ a.s., so that $E_v\left(g(X_{T_v})\mathbf{1}_{\{T_v<\infty\}}\right)\ge g(v)$ (since $g$ is increasing), which together with (\ref{v13}) and $P_v(v<X_{T_v})>0$ implies that $g'(v+)=0$ for $v>u$. So for all $v>u$,
$$
g(v)=g(\infty):=\lim_{y\to\infty}g(y).
$$
Since $g$ is nonconstant and right-continuous, we have $u>-\infty$ and $g(u)=g(\infty)$. Hence for all $x\in\mathbb{R}$, $P_x(X_{\tau_u}\ge u\;\mbox{and}\;\tau_u<\infty)=1$ and
$$
E_x\left(g(X_{\tau_u})\mathbf{1}_{\{\tau_u<\infty\}}\right)=g(\infty)\ge V(x),
$$
establishing the optimality of $\tau_u$.

We now deal with the general case $q\ge0$ and assume $E(\xi)<0$ if $q=0$. Let
\begin{equation}\label{e2.6}
\alpha=\alpha(q)=\sup\{\lambda\ge0:\phi(\lambda)=e^q\},
\end{equation}
which is positive for $q>0$. For $q=0$ (and $E(\xi)<0$ by assumption), we also have $\alpha>0$ since $\phi(\lambda)$ is convex, $\phi(0)=1$, $\phi'(0+)=E(\xi)<0$ and $\lim_{\lambda\to\infty}\phi(\lambda)=\infty$.

Suppose $-\infty<u<\infty$ and let $V^*(x)=E_{x}\left(e^{-q\tau_u}g(X_{\tau_u})
\mathbf{1}_{\left\{\tau_u<\infty\right\}}\right)$, $x\in\mathbb{R}$. Then $V^*(x)=g(x)$ for $x\ge u$. We need to prove $V^*(x)=V(x)$ for all $x$. Since $V(x)\ge V^*(x)$, it remains to show $V^*(x)\ge V(x)$. By Lemma \ref{lb}, it suffices to prove
\begin{equation}\label{e2.9}
V^*(x)\ge g(x)\;\;\mbox{for}\;\;x<u,
\end{equation}
and
\begin{equation}\label{e2.10}
V^*(x)\ge E(e^{-q}V^*(x+\xi))\;\;\mbox{for all}\;\;x.
\end{equation}
Recall that $u\ge x_0=\inf\{s:g(s)>0\}$, implying that $g(y)>0$ for all $y>u$ and $E_u\left(e^{-qT_u}g(X_{T_u})\mathbf{1}_{\left\{T_u<\infty\right\}}\right)>0$ (since $P_u(X_{T_u}>u,T_u<\infty)\ge P(\xi>0)>0$). By (\ref{g1}), we have
$$
\frac{E_y\left(e^{-qT_y}g(X_{T_y})\mathbf{1}_{\left\{T_y<\infty\right\}}\right)}{g(y)}\le1\;\;\mbox{for}\;\;y>u,
$$
which yields that
\begin{equation}\label{e2.20}
\frac{E_u\left(e^{-qT_u}g(X_{T_u})\mathbf{1}_{\left\{T_u<\infty\right\}}\right)}{g(u)}=\lim_{y\to u+}\frac{E_y\left(e^{-qT_y}g(X_{T_y})\mathbf{1}_{\left\{T_y<\infty\right\}}\right)}{g(y)}\le1,
\end{equation}
where the equality follows from $g(u)=g(u+)$ and
$$
E_y\left(e^{-qT_y}g(X_{T_y})\mathbf{1}_{\left\{T_y<\infty\right\}}\right)=E\left(e^{-qT_0}g(y+X_{T_0})\mathbf{1}_{\left\{T_0<\infty\right\}}\right)
$$
decreases to $E_u\left(e^{-qT_u}g(X_{T_u})\mathbf{1}_{\left\{T_u<\infty\right\}}\right)>0$ as $y\downarrow u$. Thus, $g(u)>0$. Noting that $g(X_{\tau_u})\ge g(u)$ on $\{\tau_u<\infty\}$, we have
$$
V^*(x)=E_x\left(e^{-q\tau_u}g(X_{\tau_u})
\mathbf{1}_{\left\{\tau_u<\infty\right\}}\right)\ge g(u)E_x\left(e^{-q\tau_u}
\mathbf{1}_{\left\{\tau_u<\infty\right\}}\right)>0\;\;\mbox{for all}\;\;x.
$$
(If $x<u$, then $E_x\left(e^{-q\tau_u}
\mathbf{1}_{\left\{\tau_u<\infty\right\}}\right)>0$ by the assumption that $P(\xi>0)>0$.) If $u=x_0$, then $V^*(x)>0=g(x)$ for $x<u$. If $u>x_0$, we have by (\ref{e2.17}) that $E_u\left(e^{-qT_u}g(X_{T_u})\mathbf{1}_{\{T_u<\infty\}}\right)/g(u)=1$, which together with Lemma \ref{lc}(ii) implies that $g(x)\le E_x\left(e^{-q\tau_u}g(X_{\tau_u})
\mathbf{1}_{\left\{\tau_u<\infty\right\}}\right)=V^*(x)$ for $x<u$. This proves (\ref{e2.9}).

To prove (\ref{e2.10}), we have by Lemma \ref{lg} that for $x<u$,
$$
V^*(x)=E_x\left(e^{-q\tau_u}g(X_{\tau_u})\mathbf{1}_{\{\tau_u<\infty\}}\right)=E_x\left(e^{-qT_u}g(X_{T_u})\mathbf{1}_{\{T_u<\infty\}}\right)=E(e^{-q}V^*(x+\xi)).
$$
By (\ref{e2.20}), we also have
\begin{equation}\label{v14}
\frac{E(e^{-q}V^*(u+\xi))}{V^*(u)}=\frac{E_u\left(e^{-qT_u}g(X_{T_u})\mathbf{1}_{\left\{T_u<\infty\right\}}\right)}{g(u)}\le1,
\end{equation}
implying that $V^*(u)\ge E(e^{-q}V^*(u+\xi))$. (Note that the inequality in (\ref{v14}) cannot be replaced by an equality if $u=x_0$.)

It remains to show (\ref{e2.10}) for $x>u$ $(>-\infty)$. Fix an (arbitrary) $x>u\;(\ge x_0)$ with $c:=h'(x-)$. We have $g(x)>0$ and $0\le c<\infty$. It follows from the concavity of $h$ that
$$
h(y)\le h(x)+c(y-x)\;\;\mbox{for all}\;\;y\in\mathbb{R}.
$$
For each $w\in[0,\infty)$, define a function $h_w:\mathbb{R}\to\mathbb{R}$ by
\begin{equation}\label{e3}
h_w(y)=
\begin{cases}
h(x)+c(y-x),&\mbox{if}\;\;y\le x+w;\\
h(y)+h(x)+cw-h(x+w),&\mbox{if}\;\;y>x+w.
\end{cases}
\end{equation}
It is readily seen that $h_w(y)\ge h(y)$ for all $y$, $h_w(x)=h(x)$ and $h_w(\cdot)$ is increasing, concave and continuous. For $0\le w_1<w_2$,
$$
h_{w_2}(y)-h_{w_1}(y)=
\begin{cases}
0,&\;\;\mbox{if}\;y\le x+w_1;\\
h(x+w_1)+c(y-x-w_1)-h(y),&\;\;\mbox{if}\;\;x+w_1<y\le x+w_2;\\
h(x+w_1)+c(w_2-w_1)-h(x+w_2),&\;\;\mbox{if}\;\;y>x+w_2.
\end{cases}
$$
By the concavity of $h$, we have for all $y\in\mathbb{R}$
\begin{equation}\label{v15}
0\le h_{w_2}(y)-h_{w_1}(y)\le h(x+w_1)+c(w_2-w_1)-h(x+w_2),
\end{equation}
so that $h_w(\cdot)$ is continuous and increasing in $w$. Note that $h_\infty(y):=\lim_{w\to\infty}h_w(y)=h(x)+c(y-x)$ for all $y\in\mathbb{R}$.

For $w\in[0,\infty]$ and $y\in\mathbb{R}$, let $g_w(y):=e^{h_{w}(y)}$ and
\begin{equation}\label{e4.13}
U_w(y):=E_y\left(e^{-q\tau_u}g_w(X_{\tau_u})
\mathbf{1}_{\left\{\tau_u<\infty\right\}}\right).
\end{equation}
Then $g_w(y)=e^{h_w(y)}\ge e^{h(y)}=g(y)$, so
\begin{equation}\label{e4.14}
U_w(y)\ge V^*(y)\;\;\mbox{for all}\;\;y.
\end{equation}
Note that
\begin{equation}\label{e8}
\frac{E(e^{-q}g_{\infty}(y+\xi))}{g_{\infty}(y)}=e^{-q}E(e^{h_{\infty}(y+\xi)-h_{\infty}(y)})=e^{-q}E(e^{c\xi})=e^{-q}\phi(c).
\end{equation}
If $c\le\alpha$, then we have by the convexity of $\phi$ and (\ref{e2.6}) that $e^{-q}\phi(c)\le e^{-q}\phi(\alpha)=1$, implying by (\ref{e8}) that $E(e^{-q}g_{\infty}(y+\xi))\le g_\infty(y)$ for $y\in\mathbb{R}$. Thus, $\{e^{-qn}g_{\infty}(X_n)\}_{n\ge0}$ (with $X_0=x$) is a (positive) supermartingale, so that $E_x\left(e^{-qT_u}g_{\infty}(X_{T_u})\mathbf{1}_{\left\{T_u<\infty\right\}}\right)\le g_{\infty}(x)=g(x)=V^*(x)$. Therefore,
\begin{align*}
\frac{E(e^{-q}V^*(x+\xi))}{V^*(x)}&\le\frac{E(e^{-q}U_{\infty}(x+\xi))}{g_{\infty}(x)}\;\;\mbox{(by (\ref{e4.14}))}\\
&=\frac{E_x\left(e^{-qT_u}g_{\infty}(X_{T_u})\mathbf{1}_{\left\{T_u<\infty\right\}}\right)}{g_{\infty}(x)}\;\;\mbox{(by Lemma \ref{lg} and (\ref{e4.13}))}\\
&\le1,
\end{align*}
proving (\ref{e2.10}) for $x>u>-\infty$ with $c\le\alpha$. (Recall that $c=h'(x-)$ and $\alpha$ is given in (\ref{e2.6}).)

We now show (\ref{e2.10}) for $x>u>-\infty$ with $c>\alpha>0$. For $w\in[0,\infty]$, let
\begin{equation}\label{e4.15}
f(w):=\frac{E_x\left(e^{-qT_x}g_w(X_{T_x})\mathbf{1}_{\left\{T_x<\infty\right\}}\right)}{g(x)}.
\end{equation}
Since $g_0(y)=g(y)$ for all $y\ge x$ and since $x>u$, we have by (\ref{g1}) that
$$
f(0)=\frac{E_x\left(e^{-qT_x}g_0\left(X_{T_x}\right)\mathbf{1}_{\left\{T_x<\infty\right\}}\right)}{g(x)}=\frac{E_x\left(e^{-qT_x}g\left(X_{T_x}\right)\mathbf{1}_{\left\{T_x<\infty\right\}}\right)}{g(x)}\le1.
$$
In addition, we have by (\ref{e8}) that $\frac{E(e^{-q}g_{\infty}(y+\xi))}{g_{\infty}(y)}=e^{-q}\phi(c)>e^{-q}\phi(\alpha)=1$, implying that $\{e^{-qn}g_{\infty}(X_n)\}_{n\ge0}$ (with $X_0=x$) is a submartingale, so that
\begin{equation}\label{v16}
E_x\left(e^{-qT^{n}_x}g_\infty(X_{T^{n}_x})\right)\ge g_\infty(x)=g(x)\;\;\mbox{for}\;\;n=1,2,\dots,
\end{equation}
where $T^{n}_x=\min\{T_x,n\}$. Recall that $E(\xi)<0$ if $q=0$, in which case $\lim_{n\to\infty}X_n=-\infty$ a.s. and $\lim_{n\to\infty}g_{\infty}(X_n)=0$ a.s. Since for all $n$,
$$
e^{-qT^{n}_x}g_\infty(X_{T^{n}_x})\le \max\{e^{-qn}g_\infty(x),e^{-qT_x}g_\infty(X_{T_x})\mathbf{1}_{\{T_x<\infty\}}\}\;\;\mbox{a.s.}
$$
and since as $n\to\infty$,
$$
e^{-qT^{n}_x}g_\infty(X_{T^{n}_x})\to e^{-qT_x}g_\infty(X_{T_x})\mathbf{1}_{\{T_x<\infty\}}\;\;\mbox{a.s.},
$$
we have by (\ref{e4.15}) and (\ref{v16}) that
$$
f(\infty)=\frac{E_x\left(e^{-qT_x}g_{\infty}(X_{T_x})\mathbf{1}_{\left\{T_x<\infty\right\}}\right)}{g(x)}=\frac{\lim_{n\to\infty}E_x\left(e^{-qT^{n}_x}g_\infty(X_{T^{n}_x})\right)}{g(x)}\ge1.
$$
Furthermore, by (\ref{v15}), for $0\le w_1<w_2<\infty$,
$$
1\le\frac{f(w_2)}{f(w_1)}\le e^{h(x+w_1)+c(w_2-w_1)-h(x+w_2)},
$$
implying that $f(w)$ is continuously increasing to $f(\infty)$ as $w\to\infty$. It follows from $f(0)\le 1\le f(\infty)$ that $f(w')=1$ for some $w'\in[0,\infty]$. Noting that $V^*(x)=g(x)=g_{w'}(x)$, we have
\begin{align}
\frac{E(e^{-q}V^*(x+\xi))}{V^*(x)}&\le\frac{E(e^{-q}U_{w'}(x+\xi))}{g_{w'}(x)}\;\;\mbox{(by (\ref{e4.14}))}\notag\\
&=\frac{E_x\left(e^{-qT_u}g_{w'}\left(X_{T_u}\right)\mathbf{1}_{\left\{T_u<\infty\right\}}\right)}{g_{w'}(x)}\;\;\mbox{(by Lemma \ref{lg})}\label{c2}\\
&\le\frac{E_x\left(e^{-qT_x}g_{w'}\left(X_{T_x}\right)\mathbf{1}_{\left\{T_x<\infty\right\}}\right)}{g_{w'}(x)}=f(w')=1,\notag
\end{align}
where the second inequality follows by Lemma \ref{lc}(i) applied to $g_{w'}$ (which is increasing and logconcave). This proves (\ref{e2.10}) for $x>u>-\infty$ with $c>\alpha>0$, and establishes the optimality of $\tau_u$ for the case $-\infty<u<\infty$.

To prove the optimality of $\tau_u$ for $u=-\infty$, note that $\tau_{-\infty}=0$ and $V^*(x)=g(x)$ for $x\in\mathbb{R}$. We need to prove $V(x)=g(x)$ for $x\in\mathbb{R}$. Since $u=-\infty$, we have by (\ref{g1}) that $g(x)>0$ and $E_x\left(e^{-qT_x}g(X_{T_x})\mathbf{1}_{\left\{T_x<\infty\right\}}\right)/g(x)\le1$ for all $x$. By Lemma \ref{lb}, it suffices to show that $g(x)\ge e^{-q}E(g(x+\xi))$ for all $x$. We can establish this inequality in exactly the same way that we proved $V^*(x)\ge e^{-q}E(V^*(x+\xi))$ for $x>u>-\infty$. (We need only to replace $\tau_u$ by $\tau_{-\infty}=0$ in (\ref{e4.13}) so that $U_w(y)=g_w(y)$, to replace $V^*(y)$ by $g(y)$ in (\ref{e4.14}) and to replace $T_u$ by $T_{-\infty}=1$ in (\ref{c2}).)

Finally, to show $V(x)>g(x)$ for $x<u$ and $V(x)=g(x)$ for $x\ge u$, all that remains to be done is to prove $V(x)>g(x)$ for $x<u$. For $x<u$ with $g(x)>0$, we have by (\ref{g1})
$$
g(x)<E_x\left(e^{-qT_x}g(X_{T_x})\mathbf{1}_{\{T_x<\infty\}}\right)\le V(x).
$$
For $x<u$ with $g(x)=0$, $g(x)<V(x)$ holds trivially since $P(\xi>0)>0$ implies $V(y)>0$ for all $y\in\mathbb{R}$. The proof of Lemma \ref{la} is complete.
\end{proof}

\begin{proof}[\bf Proof of Lemma \ref{le}]
Note that for $x\ge x'$, $E_x(e^{-qT_x}g(X_{T_x})\mathbf{1}_{\{T_x<\infty\}})\le\sup_{y\in\mathbb{R}}g(y)=c=g(x)$, i.e. $E_x(e^{-qT_x}g(X_{T_x})\mathbf{1}_{\{T_x<\infty\}})/g(x)\le1$, implying that $u\le x'$. For $k=1,2,\dots$, let $\{X_n^{(k)}\}_{n\ge0}$ be a random walk generated by (truncated) increments $\xi_i^{(k)}=\min\{\xi_i,k\}$, $i=1,2,\dots$. Define
$$
V_k(x)=\sup_{\tau\in\mathcal{M}^{(k)}}E_x\left(e^{-q\tau}g(X_\tau^{(k)})\mathbf{1}_{\{\tau<\infty\}}\right),\;x\in\mathbb{R},
$$
where $\mathcal{M}^{(k)}\subset\mathcal{M}$ is the class of stopping times with values in $[0,\infty]$ with respect to the filtration $\{\mathcal{F}_n^{(k)}\}_{n\ge0}$, where $\mathcal{F}_n^{(k)}=\sigma\{\xi_1^{(k)},\dots,\xi_n^{(k)}\}\subset\sigma\{\xi_1,\dots,\xi_n\}=\mathcal{F}_n$. Since $\{\xi_n\}$ is i.i.d. and $\{X^{(k)}_n\}$ is Markov, it is not difficult to see that the value function $V_k(x)$ cannot be increased by taking stopping times in $\mathcal{M}$, i.e.
\begin{equation}\label{e2.19}
V_k(x)=\sup_{\tau\in\mathcal{M}^{(k)}}E_x\left(e^{-q\tau}g(X_\tau^{(k)})\mathbf{1}_{\{\tau<\infty\}}\right)=\sup_{\tau\in\mathcal{M}}E_x\left(e^{-q\tau}g(X_\tau^{(k)})\mathbf{1}_{\{\tau<\infty\}}\right).
\end{equation}
For any $\tau\in\mathcal{M}$, we have  $e^{-q\tau}g(X^{(k)}_{\tau})\mathbf{1}_{\{\tau<\infty\}}\nearrow e^{-q\tau}g(X_\tau)\mathbf{1}_{\{\tau<\infty\}}$ a.s. as $k\to\infty$, so that
$$
E(e^{-q\tau}g(X^{(k)}_{\tau})\mathbf{1}_{\{\tau<\infty\}})\nearrow E(e^{-q\tau}g(X_\tau)\mathbf{1}_{\{\tau<\infty\}})\;\;\mbox{as}\;\; k\to\infty.
$$
It follows that
\begin{equation}\label{e4.8}
V_k(x)\nearrow V(x)\;\mbox{as}\;k\to\infty.
\end{equation}

Let $T_x^{(k)}=\inf\{n\ge1:X_n^{(k)}\ge x\}$ and
\begin{equation*}
u_k=\inf\left\{x\in\mathbb{R}:\frac{E_x\left(e^{-qT^{(k)}_x}g(X^{(k)}_{T^{(k)}_x})\mathbf{1}_{\{T^{(k)}_x<\infty\}}\right)}{g(x)}\le1\right\}.
\end{equation*}
Clearly, $u_k\le x'<\infty$ for all $k$. Since $E(e^{\lambda\xi^{(k)}})<\infty$  for all $\lambda\ge0$, we have by Lemma \ref{la} that
$$
V_k(x)=E_x\left(e^{-q\tau^{(k)}_{u_k}}g(X^{(k)}_{\tau^{(k)}_{u_k}})\mathbf{1}_{\{\tau^{(k)}_{u_k}<\infty\}}\right)>0\;\;\mbox{for all}\;\;x,
$$
where $\tau^{(k)}_{u_k}=\inf\{n\ge0:X^{(k)}_n\ge u_k\}$. Since $V_{k+1}(x)\ge V_k(x)$ for all $x$, we have $g(u_{k+1})=V_{k+1}(u_{k+1})\ge V_k(u_{k+1})\ge g(u_{k+1})$. So $g(u_{k+1})=V_{k}(u_{k+1})$, implying that $u_{k+1}\ge u_k$.

Let $u_\infty=\lim_{k\to\infty}u_k$. We prove the optimality of $\tau_u$ in two steps. We show in step 1 that $\tau_{u_\infty}$ is optimal and in step 2 that $u_\infty=u$.

{\it Step 1.} We want to show
\begin{equation}\label{e10}
V(x)=E_x(e^{-q\tau_{u_\infty}}g(X_{\tau_{u_\infty}})\mathbf{1}_{\{\tau_{u_\infty}<\infty\}})\;\;\mbox{for all}\;\;x,
\end{equation}
which implies that $\tau_{u_\infty}$ is optimal. Since $u_k\nearrow u_\infty$, we have for any $x\ge u_\infty$ that $x\ge u_k$ and $V_k(x)=g(x)$ for all $k$, implying by (\ref{e4.8}) that
$$
V(x)=\lim_{k\to\infty}V_k(x)=g(x)=E_x\left(e^{-q\tau_{u_\infty}}g(X_{\tau_{u_\infty}})\mathbf{1}_{\{\tau_{u_\infty}<\infty\}}\right),
$$
establishing (\ref{e10}) for $x\ge u_\infty$. It remains to prove (\ref{e10}) for $x<u_\infty$. It suffices to show for $x<u_\infty$ that
\begin{equation}\label{e11}
V_k(x)=E_x(e^{-q\tau^{(k)}_{u_k}}g(X^{(k)}_{\tau^{(k)}_{u_k}})\mathbf{1}_{\{\tau^{(k)}_{u_k}<\infty\}})\to E_x(e^{-q\tau_{u_\infty}}g(X_{\tau_{u_\infty}})\mathbf{1}_{\{\tau_{u_\infty}<\infty\}})\;\;\mbox{as}\;\;k\to\infty.
\end{equation}
We argue below that with $X_0=x<u_\infty$, as $k\to\infty$
\begin{equation}\label{e12}
e^{-q\tau^{(k)}_{u_k}}g(X^{(k)}_{\tau^{(k)}_{u_k}})\mathbf{1}_{\{\tau^{(k)}_{u_k}<\infty\}}\to e^{-q\tau_{u_\infty}}g(X_{\tau_{u_\infty}})\mathbf{1}_{\{\tau_{u_\infty}<\infty\}}\;\;\mbox{a.s.},
\end{equation}
which together with the bounded convergence theorem implies (\ref{e11}). We now show that (\ref{e12}) holds a.s. on $\{\tau_{u_\infty}<\infty\}$ and on $\{\tau_{u_\infty}=\infty\}$ separately. Let $\Delta=\max_{j\le \tau_{u_\infty}}\xi_j<\infty$ on $\{\tau_{u_\infty}<\infty\}$. Then, for $k>\Delta$, $X^{(k)}_i=\sum_{j=1}^{i}\xi^{(k)}_j=\sum_{j=1}^{i}\xi_j=X_i$ for all $i\le\tau_{u_\infty}$. Letting $\Delta'=\max_{0\le i<\tau_{u_\infty}}X_i$ ($<u_\infty$), choose a sufficiently large $k_0$ such that $k_0>\Delta$ and $u_{k_0}>\Delta'$. Thus, $\tau^{(k)}_{u_k}=\tau_{u_\infty}$ and $X^{(k)}_{\tau^{(k)}_{u_k}}=X_{\tau_{u_\infty}}$ for all $k\ge k_0$, implying that
$$
e^{-q\tau^{(k)}_{u_k}}g(X^{(k)}_{\tau^{(k)}_{u_k}})\mathbf{1}_{\{\tau^{(k)}_{u_k}<\infty\}}= e^{-q\tau_{u_\infty}}g(X_{\tau_{u_\infty}})\mathbf{1}_{\{\tau_{u_\infty}<\infty\}}\;\;\mbox{for all}\;\;k\ge k_0.
$$
So (\ref{e12}) holds a.s. on $\{\tau_{u_\infty}<\infty\}$.

To show (\ref{e12}) holds a.s. on $\{\tau_{u_\infty}=\infty\}$, note that the random walk $\{X_0,X_1,\dots\}$ either satisfies $\overline{\lim}_{n\to\infty}X_n=+\infty$ a.s. or $\lim_{n\to\infty}X_n=-\infty$ a.s. If $\overline{\lim}_{n\to\infty}X_n=+\infty$ a.s., then $\tau_{u_\infty}<\infty$ a.s., so that trivially (\ref{e12}) holds a.s. on $\{\tau_{u_\infty}=\infty\}$. Now suppose $\lim_{n\to\infty}X_n=-\infty$ a.s. Then on $\{\tau_{u_\infty}=\infty,\lim_{n\to\infty}X_n=-\infty\}$, there is an $n_0\in\mathbb{N}$ such that $X_n<u_\infty$ for $0\le n\le n_0$ and $X_n<u_\infty-1$ for $n>n_0$. Let $0<\varepsilon<1$ be such that $\varepsilon<u_\infty-\max\{X_0,X_1,\dots,X_{n_0}\}$, so that $X_n<u_\infty-\varepsilon$ for all $n$. Choose $k_1$ such that $u_k>u_\infty-\varepsilon$ for all $k\ge k_1$. For $k\ge k_1$, $X_n^{(k)}\le X_n<u_\infty-\varepsilon<u_k$ for all $n$, so that $\tau^{(k)}_{u_k}=\infty$ for $k\ge k_1$. Hence, $e^{-q\tau^{(k)}_{u_k}}g(X^{(k)}_{\tau^{(k)}_{u_k}})\mathbf{1}_{\{\tau^{(k)}_{u_k}<\infty\}}=0$ for $k\ge k_1$. Since the right-hand side of (\ref{e12}) is 0 on $\{\tau_{u_\infty}=\infty\}$, it follows that (\ref{e12}) holds a.s. on $\{\tau_{u_\infty}=\infty\}$. This completes step 1.

{\it Step 2.} We now prove that $u_\infty=u$. If $u_\infty=-\infty$, then $x_0=-\infty$, i.e. $g(x)>0$ for all $x\in\mathbb{R}$. By the optimality of $\tau_{u_\infty}=\tau_{-\infty}=0$,
$$
g(x)=V(x)\ge E_x\left(e^{-qT_x}g(X_{T_x})\mathbf{1}_{T_x<\infty}\right)\;\;\mbox{for all}\;\;x\in\mathbb{R},
$$
implying that $u=-\infty$.

If $u>u_\infty>-\infty$, we have by (\ref{g1})
$$
E_{u_\infty}(e^{-qT_{u_\infty}}g(X_{T_{u_\infty}})\mathbf{1}_{\{T_{u_\infty}<\infty\}})/g(u_\infty)>1,
$$
(noting that $g(u_\infty)=V(u_\infty)>0$ by (\ref{e10})). Then we have
$$
g(u_\infty)<E_{u_\infty}(e^{-qT_{u_\infty}}g(X_{T_{u_\infty}})\mathbf{1}_{\{T_{u_\infty}<\infty\}})\le V(u_\infty)=g(u_\infty),
$$
a contradiction. So $u\le u_\infty$. Now suppose $u<u_\infty$. Let $x$ be such that $u<x<u_\infty$. By (\ref{g1}),
$$
\frac{E_x(e^{-qT_x}g(X_{T_x})\mathbf{1}_{\{T_x<\infty\}})}{g(x)}\le1,
$$
which by Lemma \ref{ld}(ii), implies that
\begin{equation}\label{e4.10}
g(x)\ge E_x(e^{-q\tau_{u_\infty}}g(X_{\tau_{u_\infty}})\mathbf{1}_{\{\tau_{u_\infty}<\infty\}})=V(x).
\end{equation}
Since $x<u_\infty$, we can choose $k$ such that $x<u_k$. By Lemma \ref{la} applied to $u_k$ and $V_k$, we have $g(x)<V_k(x)\le V(x)$, contradicting (\ref{e4.10}). This proves $u=u_\infty$ and establishes the optimality of $\tau_u$.

Finally, we show $V(x)>g(x)$ for $x<u$ and $V(x)=g(x)$ for $x\ge u$. It follows from (\ref{e10}) that $V(x)=g(x)$ for $x\ge u=u_\infty$. For $x<u=u_\infty$, choose a large $k$ so that $x<u_k$. By Lemma \ref{la} applied to $u_k$ and $V_k$, we have $g(x)<V_k(x)\le V(x)$. The proof is complete.

\end{proof}

\section*{Acknowledgements}
\hspace*{18pt}The authors gratefully acknowledge support from the Ministry of Science and Technology of Taiwan, ROC.


\end{document}